\pdfoutput=1
\documentclass[11pt]{article}
\usepackage[utf8]{inputenc}
\usepackage{amsmath,amssymb,amsthm}

\usepackage{mathtools}
\usepackage{graphicx}
\usepackage[margin=1in]{geometry}
\usepackage{placeins}
\newtheorem{theorem}{Theorem}
\newtheorem{proposition}{Proposition}
\newtheorem{corollary}{Corollary}
\newtheorem{lemma}{Lemma}
\newtheorem{conjecture}{Conjecture}
\theoremstyle{definition}
\newtheorem{definition}{Definition}
\newtheorem{assumption}{Assumption}
\newtheorem{openproblem}{Open Problem}
\theoremstyle{remark}
\newtheorem{remark}{Remark}

\newcommand{\R}{\mathbb{R}}
\newcommand{\E}{\mathbb{E}}
\newcommand{\PP}{\mathcal{P}_2(\R^d)}
\newcommand{\Ck}{\mathcal{C}_k(\varepsilon)}
\newcommand{\mean}{\mathrm{mean}}
\newcommand{\KL}{\mathrm{KL}}
\newcommand{\TV}{\mathrm{TV}}

\title{\vspace{-2em}A Temporal--Spatial Minimax Rate for Smoothly-Varying Distributions in Wasserstein Space}
\author{Munsik Kim\\ \texttt{physicist456@gmail.com}}
\date{June 2026}

\begin{document}
\maketitle
\vspace{-1.5em}

\begin{abstract}\noindent
We study the minimax rate of estimating a future value $\mu_{t_n+h}$ of a curve $t\mapsto\mu_t$ in
the $2$-Wasserstein space $\PP$ from finitely many noisy snapshots of its past, under an adiabatic
bound $\|\nabla_t^k v\|\le\varepsilon$ on the $k$-th covariant derivative of the velocity field. Our
central result is a \emph{unified temporal--spatial minimax lower bound}: over regular, locally
transport-rich subclasses of $\PP$, every estimator incurs $W_2$-risk with $M$-exponent
$\gamma_d(k+1)/(k+1+\gamma_d)$, $\gamma_d=\min(1/d,1/2)$ ($M$ the total sample size). It follows from
a \emph{temporal-to-spatial reduction}: the smoothness budget defines a reachable $W_2$-ball into
which a transport packing is embedded along the time axis, and the information of the \emph{entire}
snapshot experiment is controlled by a Fano argument --- the spatial packing is classical, but its
smoothness-admissible temporal embedding and the full-window analysis are new. The bound interpolates
a dimension-free extrapolation floor of order $\varepsilon h^{k+1}$ --- the irreducible cost of an
unobserved future, present even with the exact past --- and the spatial estimation curse
$M^{-\gamma_d}$, recovering the static distribution-estimation rate as $k\to\infty$. We state the
lower bound in a design-dependent form --- with a design-weighted effective sample size --- valid for
arbitrary observation times, and obtain the closed-form exponent in the dense (equispaced) regime. The matching upper bound is established at
$k=0$ (rate $M^{-1/(d+1)}$, $d\ge3$) and, in a translation submodel, for all $k$; for $k\ge1$ a
covariant estimator attains the rate \emph{conditionally} on two estimates (a comparison-geometry
bias bound and an optimal-transport map-estimation rate), leaving the unconditional general-$k$ upper
bound as an explicit open problem. Numerical experiments on synthetic curved and flat families
corroborate the predicted exponents.
\end{abstract}

\section{Introduction}

Many systems are described not by a state but by a \emph{distribution} that drifts in time: the
cross-sectional distribution of incomes or firm sizes from year to year, the law of a particle
ensemble under a slowly changing potential, a population of single cells across developmental
time, the distribution of intraday returns from one day to the next, or the law of the hidden
states and token embeddings as a sequence model runs. In each case one observes a trajectory
$t\mapsto\mu_t$ of probability measures and would like to \emph{forecast} the measure
$\mu_{t_n+h}$ a horizon $h$ beyond the last observation. A large and active literature builds
estimators for this task --- Wasserstein autoregression, Koopman-operator models,
distribution-on-distribution regression --- yet a prior question is left open: \emph{how far
ahead is an evolving distribution forecastable at all, and what sets the limit?}

We study this question in the $2$-Wasserstein space $\PP$ with its Otto--Benamou--Brenier
Riemannian structure, the natural geometry for measures that move by transport. The single
assumption is regularity in time: the velocity field $v_t$ of the curve has a bounded $k$-th
covariant derivative, $\|\nabla_t^k v_t\|\le\varepsilon$, defining an \emph{adiabatic}
(slowly-varying) class $\Ck$. The index $k$ graduates the assumption --- $k=0$ bounds the speed,
$k=1$ the acceleration (a near-geodesic curve), $k=2$ a near-spline --- the Wasserstein analogue
of assuming a forecastable scalar signal has $k$ bounded derivatives, the setting of classical
extrapolation theory.

\paragraph{A tension between time and space.} Forecasting an evolving distribution is governed by
two opposing forces. \emph{Temporal smoothness helps:} a curve with a bounded $(k{+}1)$-st
derivative can be extrapolated by a Taylor/geodesic rule whose error grows only as $h^{k+1}$, so
more controlled derivatives buy a longer forecastable horizon. \emph{Spatial dimension hurts:} a
measure on $\R^d$ must be learned from finitely many samples, and even the empirical measure
converges in $W_2$ only at the dimension-cursed rate $M^{-1/d}$ for $d>2$ (with parametric saturation $M^{-1/2}$ in low dimension). The central object of this paper
is how these forces combine into a single forecasting limit, and the phase structure ---
extrapolation-limited versus statistics-limited --- that results.

\paragraph{Contributions.} We prove a hierarchy of lower bounds, each matched by an explicit
forecaster to the extent stated, and are deliberate about which half --- lower bound, upper
bound, or constant --- is established. The \emph{centerpiece} is the unified temporal--spatial
minimax lower bound (Theorem~\ref{thm:unified}); its temporal-to-spatial reduction
(Lemma~\ref{lem:reach}, Appendix~\ref{app:reduction}) is the paper's main technical step, and the
exact-past and location-channel results below are both its ingredients and results in their own
right. Table~\ref{tab:summary} states exactly what is proven, conditional, or open.
\begin{itemize}
\item \textbf{Exact-past floors.} Even given the entire past exactly, no forecaster beats a
worst-case floor $\varepsilon h^{k+1}/(k+1)!$ (Theorem~\ref{thm:lecam}) or, under a process
prior, a complementary average-case floor $d\,\sigma^2 h^{2k+1}/((k!)^2(2k+1))$
(Theorem~\ref{thm:bayes}); the order-$k$ geodesic/spline extrapolator attains the worst-case
scaling, with the \emph{exact} constant on flat translation submodels
(Proposition~\ref{prop:sharp}).
\item \textbf{A clean statistical separation.} With $M$ samples the limit splits into a
dimension-free extrapolation floor and a dimension-cursed shape floor that do not interact
(Theorem~\ref{thm:stat}); in the \emph{location channel} the latter sharpens to
$M^{-(k+1)/(2k+3)}$ (Theorem~\ref{thm:sharp}), the classical H\"older rate
$M^{-\beta/(2\beta+1)}$, $\beta=k+1$, lifted to distribution forecasting.
\item \textbf{A unified temporal--spatial rate.} Over regular, locally transport-rich subclasses of $\PP$ we prove a lower bound, valid for all
$k$, with $M$-exponent $\gamma_d(k+1)/(k+1+\gamma_d)$, $\gamma_d=\min(1/d,1/2)$
(Theorem~\ref{thm:unified}), interpolating smoothness and the spatial curse and recovering the
static shape rate $M^{-\gamma_d}$ as $k\to\infty$; the temporal-to-spatial reduction behind it ---
turning the smoothness budget into a reachable $W_2$-ball and embedding a transport packing along
the time axis (Lemma~\ref{lem:reach}, Appendix~\ref{app:reduction}) --- is the central technical
step. The matching upper bound is established at $k=0$
(Theorem~\ref{thm:upper0}, rate $M^{-1/(d+1)}$ for $d\ge3$), sharp there; for $k\ge1$ we give a
covariant (development-based) forecaster that attains it \emph{conditionally} on two estimates --- a
comparison-geometry bias bound and an optimal-transport map-estimation rate
(Proposition~\ref{prop:covariant}). The \emph{unconditional} matching upper bound for $k\ge1$ remains
Conjecture~\ref{conj:upper}.
\item \textbf{Phase structure and verification.} The bounds predict a sharp boundary between
extrapolation- and statistics-limited regimes; numerics confirm the integer horizon exponents and
the $(N,h)$ phase boundary, track the unified exponent's spatial curse out to $d=6$, and show ---
end-to-end on a morphing nonparametric shape --- that a moving distribution is forecastable strictly
more slowly than a static one, as the unified rate predicts.
\end{itemize}

\begin{table}[t]
\centering\small
\begin{tabular}{lcccc}
\hline
Result & Range & Lower & Upper & Constant \\
\hline
Exact-past floor $\varepsilon h^{k+1}$ & all $k$ & proven & submodels & exact (flat) \\
Average-case floor $\sigma^2 h^{2k+1}$ & all $k$ & proven & submodels & exact (Gaussian) \\
Location channel $M^{-(k+1)/(2k+3)}$ & all $k$ & proven & proven & rate \\
Unified $M^{-\gamma_d(k+1)/(k+1+\gamma_d)}$ & $k=0$ & proven & proven & rate \\
Unified (same exponent) & $k\ge1$ & proven & conditional & open \\
\hline
\end{tabular}
\caption{What is established. ``Submodels'': the matching upper bound holds on flat translation
(and Gaussian) submodels. ``Conditional'': Proposition~\ref{prop:covariant} attains the rate under
the comparison-geometry and map-estimation estimates~(C),(S); the unconditional statement is
Conjecture~\ref{conj:upper}.}
\label{tab:summary}
\end{table}

\paragraph{Technique.} Two reductions carry the lower bounds. A translation embedding realizes
$\R^d$ as a flat, totally geodesic submanifold of $\PP$ (Lemma~\ref{lem:embed}), reducing the
exact-past and location-channel bounds to scalar extrapolation and Le~Cam/van~Trees two-point
arguments. For the unified bound a \emph{reachability lemma} (Lemma~\ref{lem:reach}) shows that
any measure within a smoothness-budgeted $W_2$-ball is reached by an admissible curve, obtained by
reparametrizing a $W_2$-geodesic in time; this turns temporal forecasting into spatial estimation
from the pooled in-window samples, where the empirical-$W_2$ minimax rate enters. The classical
empirical-$W_2$ exponent is recovered, not assumed: Appendix~\ref{app:reduction} constructs the
spatial packing explicitly; the temporal reduction and the reachability construction are ours.

\paragraph{Scope and outline.} Section~\ref{sec:stat} contains the finite-sample theory and
Section~\ref{sec:unified} the unified rate; the remaining sections give the exact-past floors, the
$k=0$ matching upper bound (with the conditional $k\ge1$ construction in
Appendix~\ref{app:covariant}), and a numerical illustration. Proofs are deferred to
Appendix~\ref{app:proofs}; Appendix~\ref{app:forecaster} details the degree-$k$ forecaster.

\section{Related work}

\paragraph{Distributional and functional time series.} Forecasting a measure-valued trajectory is
an active methodological area. Wasserstein autoregression models density time series in the
tangent space of $\PP$ (Zhang--Kokoszka--Petersen~\cite{war}), with inferential and diagnostic
extensions; Koopman-operator methods lift the dynamics to a linear evolution on observables
(Wang--Araki~\cite{kw}); and distribution-on-distribution regression learns transport maps between
measures (Ghodrati--Panaretos~\cite{gp}). These sit within functional data analysis and functional
time series more broadly. All propose \emph{estimators}; we instead ask for the \emph{limits} that
bound any such method under a smoothness-only assumption --- a complementary and, to our knowledge,
previously unaddressed question for Wasserstein forecasting.

\paragraph{Trajectory inference and population dynamics.} Reconstructing how a distribution evolves
from temporal snapshots is central to single-cell genomics, where optimal-transport methods recover
developmental trajectories (Schiebinger et al.~\cite{schiebinger}) and a mathematical theory of
trajectory inference has emerged (Lavenant et al.~\cite{lavenant}). That line \emph{interpolates} a
population between observed times; we study the distinct, harder problem of \emph{extrapolating} it
beyond the last observation, and the fundamental limit on doing so.

\paragraph{Geometry of Wasserstein space.} Our regularity class and extrapolators rest on the
Otto calculus of $\PP$ (Otto~\cite{otto}; Benamou--Brenier~\cite{bb};
Ambrosio--Gigli--Savar\'e~\cite{ags}; Villani~\cite{villani}) and its second-order theory
(Gigli~\cite{gigli}), in particular geodesics and the covariant derivative of velocity fields.
Smooth interpolation on Wasserstein space --- splines and higher-order models
(Benamou--Gallou\"et--Vialard~\cite{bgv}; Chewi et al.~\cite{chewi}) --- provides the $k\ge1$
extrapolators whose forecast error we bound, and we quantify when curvature corrections enter
(Proposition~\ref{prop:sharp}).

\paragraph{Empirical measures in Wasserstein distance.} The spatial half of our rates is governed
by convergence of the empirical measure in $W_2$, classical since Dudley~\cite{dudley} and
sharpened by Fournier--Guillin~\cite{fg} and Weed--Bach~\cite{weedbach}, with matching minimax
density-estimation rates by Niles-Weed--Berthet~\cite{nwb} and Singh--P\'oczos~\cite{sp}. We invoke
these at the optimized temporal bandwidth; the dimension curse they exhibit is exactly what
degrades the forecast exponent away from the parametric rate.

\paragraph{Nonparametric estimation and extrapolation.} The location channel reproduces the
classical minimax theory of nonparametric estimation under H\"older smoothness
(Stone~\cite{stone}; Tsybakov~\cite{tsybakov}): the rate $M^{-\beta/(2\beta+1)}$, $\beta=k+1$,
appears here as a forecasting lower bound --- the extrapolation (boundary) instance of that theory.
We make this correspondence explicit rather than presenting the rate as new.

\paragraph{Prediction under temporal change.} Forecasting a smooth signal from its past is the
classical Kolmogorov--Wiener problem; our exact-past floors are its distributional,
finite-smoothness analogue. Statistically, slow variation in time is the premise of locally
stationary processes (Dahlhaus~\cite{dahlhaus}) and, in machine learning, of learning under
concept drift (Gama et al.~\cite{gama}); our bounds quantify the price such drift imposes on
distributional forecasting, with the optimal pooling bandwidth emerging from a bias--variance
balance.

\section{Setup}

Let $(\PP, W_2)$ carry the Otto--Benamou--Brenier formal Riemannian structure. We
observe $t \mapsto \mu_t$ on $[0,t_n]$ and forecast $\mu_{t_n+h}$, $h>0$. A
\emph{forecaster} is any measurable $\hat\nu = \hat\nu(\mu|_{[0,t_n]}) \in \PP$. For
an absolutely continuous curve the velocity field $v_t$ satisfies
$\partial_t \mu_t + \nabla\cdot(\mu_t v_t)=0$; $\nabla_t$ denotes the covariant
derivative of vector fields along the curve (Gigli's second-order calculus, assuming
the requisite tangent-module regularity), with $\|\cdot\|_{\mu_t}$ the $L^2(\mu_t)$
tangent norm.

\begin{definition}[Order-$k$ slow-variation class]\label{def:class}
For $k\ge 0$, $\varepsilon>0$, let $\Ck$ be the absolutely continuous curves whose
velocity field admits covariant derivatives up to order $k$ with
$\operatorname*{ess\,sup}_{t}\| \nabla_t^{\,k} v_t\|_{\mu_t}\le \varepsilon$,
$v_t=\dot\mu_t$. Informally the curve's $(k{+}1)$-st covariant derivative is bounded
by $\varepsilon$ (geodesic for $k{=}1$, cubic spline for $k{=}2$).
\end{definition}

\paragraph{Two elementary facts.} Both deterministic bounds reduce to a shift problem.

\begin{lemma}[Isometric translation embedding]\label{lem:embed}
Fix $\rho\in\PP$, $\tau_x(z)=z+x$. Then $x \mapsto (\tau_x)_\#\rho$ is an isometric
embedding of $(\R^d,|\cdot|)$ into $(\PP,W_2)$; its image is totally geodesic and flat.
\end{lemma}

\begin{lemma}[Mean contraction]\label{lem:mean}
$W_2(\alpha,\beta)\ge|\mean(\alpha)-\mean(\beta)|$ for all $\alpha,\beta\in\PP$.
\end{lemma}

\section{Worst-case (Le Cam) lower bound}

\begin{theorem}[Extrapolation floor, worst case]\label{thm:lecam}
For every integer $k\ge 0$, $h>0$, $\varepsilon>0$,
\[
  \inf_{\hat\nu}\ \sup_{\mu_\bullet\in\Ck}\ W_2(\hat\nu,\mu_{t_n+h})\ \ge\ \frac{\varepsilon\,h^{k+1}}{(k+1)!}.
\]
\end{theorem}
\begin{lemma}[Mollification]\label{lem:moll}
Convolving $x_b^{(k+1)}$ with a width-$\eta$ kernel makes the construction $C^{k+1}$
with $\|\nabla^k v\|\le\varepsilon$ and changes the separation by $1+O(\eta/h)$.
\end{lemma}

\section{Average-case lower bound (conditional Bayes floor)}

\begin{theorem}[Conditional Bayes floor]\label{thm:bayes}
Fix $\rho$ and let $\mu_t=(\tau_{x(t)})_\#\rho$ with $x^{(k+1)}(t)=\xi(t)$, $\xi$ white
noise of intensity $\sigma^2 I_d$ (a $(k{+}1)$-fold integrated Brownian motion). Then
for every forecaster measurable w.r.t.\ $\{\mu_u:u\le t_n\}$,
\[
   \E\, W_2^2(\hat\nu,\mu_{t_n+h})\ \ge\ \frac{d\,\sigma^2\,h^{2k+1}}{(k!)^2(2k+1)},
\]
with equality, within this model, for the order-$k$ Taylor extrapolator (the
conditional mean). This stochastic model is the \emph{average-case analogue} of $\Ck$: its $(k{+}1)$-st
derivative has \emph{variance} $\sigma^2$ rather than an almost-sure sup bound, so it is not
contained in the deterministic class. The worst-case floor of Theorem~\ref{thm:lecam} and this
average-case floor are therefore complementary, not nested.
\end{theorem}

This floor is the irreducible \emph{process noise} between $t_n$ and $t_n+h$; no amount
of memory or data removes it.

\section{Matching upper bound and the role of curvature}

\begin{lemma}[Curvature-controlled remainder; rigorous on flat/finite-dim submodels]
\label{lem:curv}
Let $\mu_\bullet$ be $C^{k+1}$ with well-defined endpoint $k$-jet and $P_k$ its order-$k$
extrapolator. On the flat translation submodel,
$W_2(\mu_{t_n+h},P_k(t_n+h))=\frac{h^{k+1}}{(k+1)!}\|\nabla_t^{k}v_{t_n}\|_{\mu_{t_n}}$
exactly. On a finite-dimensional totally geodesic submanifold (e.g.\ centered Gaussians
under the Bures metric) a normal-coordinate computation gives the same leading term for
$k\le 2$, curvature entering only at order $h^{k+3}$ via the identity
$(\partial_m\Gamma^i_{jk})(p)v^mv^jv^k=0$ (Riemann antisymmetry against the symmetric
$v^{\otimes3}$). On general $\PP$ this expansion is a \emph{formal} Otto-calculus
computation.
\end{lemma}

\begin{proposition}[Sharp on flat submodels; rate match on regular finite-dimensional submanifolds]\label{prop:sharp}
On the flat translation submodel, and on finite-dimensional totally geodesic Wasserstein
submanifolds satisfying the normal-coordinate hypotheses of Lemma~\ref{lem:curv} (e.g.\ centered
Gaussians under the Bures metric, $k\le 2$), the order-$k$ extrapolator attains
$\varepsilon h^{k+1}/(k+1)!+O(h^{k+2})$, matching Theorem~\ref{thm:lecam} in exponent; on flat
translation submodels the constant is in addition exact. On general $\PP$ the same leading term is
a \emph{formal} Otto-calculus expansion (Lemma~\ref{lem:curv}); for $k\ge 3$ or positively curved
submodels the exponent is expected to persist with a possibly curvature-corrected constant, not
proved here. Nonnegative Alexandrov curvature of $\PP$ is expected to give one-sided control at
finite $h$.
\end{proposition}

\section{Statistical floor: finite samples}\label{sec:stat}

Replace exact observation by $N$ i.i.d.\ samples from each of $n$ snapshots at
$t_i=t_n-(n-1-i)\Delta$, window $L=(n-1)\Delta$, total $M=Nn$. For a bandwidth $H\in(0,L]$ write
$n_H=\#\{i:t_i\in[t_n-H,t_n]\}$ for the number of in-window observation times and $M_H=N\,n_H$ for the
in-window sample count; for the equispaced design $n_H\asymp1+H/\Delta$, so $M_H\asymp MH/L$ once
$H\ge\Delta$, while $M_H=N$ for $H<\Delta$ (the window then holds only the endpoint snapshot). Let
$w=(1,h,\dots,h^k)^\top$
and $G_{jl}=\sum_{i=0}^{n-1}(t_i-t_n)^{j+l}$ the design Gram matrix.

\begin{theorem}[Statistical floor, separated from the extrapolation floor]\label{thm:stat}
If $\rho$ has Fisher information $I_e=e^\top I(\rho)e>0$ along $e$, then
\[
  \inf_{\hat\nu}\sup_{\mu_\bullet\in\Ck}\E\,W_2(\hat\nu,\mu_{t_n+h})
  \ \gtrsim\
  \max\Big\{\,
   \underbrace{c_d\,M^{-\gamma_d}}_{\textup{(A) shape: $\gamma_d=\min(1/d,1/2)$}}
   ,\
   \underbrace{\tfrac14\big[\tfrac{\varepsilon h^{k+1}}{(k+1)!}
   +(N I_e)^{-1/2}\big(w^\top G^{-1}w\big)^{1/2}\big]}_{\textup{(B) extrapolation + location leverage}}
   \Big\}.
\]
\end{theorem}

\noindent Here $\gamma_d=\min(1/d,1/2)$: the shape floor (A) is the dimension curse $M^{-1/d}$ for
$d\ge3$ and saturates to the parametric $M^{-1/2}$ in $d\le2$ (with a possible critical logarithmic
correction at $d=2$, whose exact form depends on the estimation class and on whether empirical or
optimized estimators are considered; we state (A) at the power-law level).

\begin{corollary}[Extrapolation leverage]\label{cor:leverage}
For the equispaced design and $h\gtrsim L$, $\sqrt{\frac1{N I_e}}\sqrt{w^\top G^{-1}w}
\asymp c_k(h/L)^{k}M^{-1/2}$: the parametric rate amplified by leverage $(h/L)^k$. The
governing scale is the window $L$, not the spacing $\Delta$.
\end{corollary}

\begin{theorem}[Sharp nonparametric extrapolation rate, location channel]\label{thm:sharp}
In the translation submodel with $\rho$ Gaussian (per-sample location variance
$\sigma_1^2$), $n$ equispaced snapshots over a window $L$, and $M=Nn$ total samples, the
minimax forecast error in the location channel is
\[
  \inf_{\hat\nu}\sup_{\mu_\bullet\in\Ck}\E\,W_2(\hat\nu,\mu_{t_n+h})
  \ \asymp\ \frac{\varepsilon}{(k+1)!}\,(h+H_*)^{k+1},
  \qquad
  H_*=\Big(\frac{\sigma_1^2 L}{M\varepsilon^2}\Big)^{\!\frac{1}{2k+3}},
\]
equivalently $\asymp\max\{\varepsilon h^{k+1},\,
\varepsilon^{\frac1{2k+3}}(\sigma_1^2 L/M)^{\frac{k+1}{2k+3}}\}$. The statistics-dominated
branch ($h\lesssim H_*$) is the classical H\"older-$\beta$ pointwise rate
$M^{-\beta/(2\beta+1)}$ with $\beta=k+1$; the extrapolation-dominated branch
($h\gtrsim H_*$) is the dimension-free floor of Theorem~\ref{thm:lecam}. For the equispaced design the
two-sided rate presumes the optimal width resolves the snapshots, $H_*\ge\Delta$; otherwise the
continuous-design nonparametric branch ceases to apply and the risk enters a resolution-limited regime
controlled by $\Delta$ and $N$, of rate
$\max\{\varepsilon h^{k+1},\,\min(\varepsilon\Delta^{k+1},\,\sigma_1 N^{-1/2})\}$, reducing to the
displayed $\varepsilon\Delta^{k+1}$ scale when discretization dominates the per-snapshot sampling noise
$\sigma_1 N^{-1/2}$.
\end{theorem}
\begin{remark}
This sharpens Theorem~\ref{thm:stat}(B): the degree-$k$-polynomial construction there gives
only the loose parametric floor $M^{-1/2}(h/L)^k$, whereas spending the $(k{+}1)$-st
derivative budget on a width-$H_*$ bump yields the tight nonparametric rate
$M^{-(k+1)/(2k+3)}$ (Figure~\ref{fig:sharp}). This is the classical H\"older-smoothness
nonparametric rate $M^{-\beta/(2\beta+1)}$, $\beta=k+1$ (Stone 1980; Tsybakov 2009), here
arising as a location-channel forecasting lower bound. The dimension-cursed shape channel of
Theorem~\ref{thm:stat}(A) is folded into the unified rate of Section~\ref{sec:unified}.
\end{remark}

\section{Unified temporal--spatial rate over $\PP$}\label{sec:unified}

\begin{definition}[Regular class]\label{def:reg}
The forecasting problem is \emph{regular} if the curve takes values in densities on a fixed compact
convex $\Omega\subset\R^d$, with densities bounded in $[c/2,2C]$ for fixed $0<c\le C<\infty$, with the
$W_2$-optimal maps from the reference $\mu_0$ to its members sufficiently smooth (along the
reference-centered displacement geodesics used below), the family \emph{star-geodesically closed
around} $\mu_0$
(the displacement geodesic from $\mu_0$ to each member stays in the family), and \emph{locally
transport-rich} around the reference density $\mu_0$: every sufficiently small smooth compactly
supported potential $\rho$ keeps $(\mathrm{id}+\nabla\rho)_\#\mu_0$ and its displacement
interpolation in the class. The stronger pairwise/chart regularity used \emph{only} by the
conditional upper bound of Appendix~\ref{app:covariant} --- smooth Brenier maps from the barycenter to
members --- is not part of this definition and is stated separately as Assumptions~(C),(S). This makes the tangent calculus of Lemma~\ref{lem:reach} applicable and
the packing of Appendix~\ref{app:reduction} admissible (Ambrosio--Gigli--Savar\'e; Gigli). A
\emph{single} hard reference suffices for the lower bound, and we take $\mu_0\equiv1$ to be the
\emph{uniform} (constant) density on $\Omega=[0,1]^d$. We stress that the regular class is a
\emph{local transport-rich neighborhood} of $\mu_0$ within the two-sided-bounded family $[c/2,2C]$,
\emph{not} the whole bounded-density class: not every bounded density has smooth optimal maps or a
star-geodesically closed neighborhood, and we claim neither. The perturbations of
Appendix~\ref{app:reduction} stay within $[c/2,2C]$, keep the reference-centered optimal maps smooth, and keep the
displacement interpolations in the class; the spatial constancy of $\mu_0$ --- not merely a two-sided
bound $c\le\mu_0\le C$ --- is what the separation estimates of Appendix~\ref{app:reduction} use, and
no regularity of $\mu_0$ beyond constancy is invoked.
\end{definition}

\begin{proposition}[A nontrivial reference-star regular class]\label{prop:regexists}
Fix $\mu_0\equiv1$ on $\Omega=[0,1]^d$. There is $r_0>0$ such that the transport neighborhood
\[
   \mathcal F_{r_0}=\big\{(\mathrm{id}+\nabla\rho)_\#\mu_0:\ \rho\in C_c^\infty(\Omega),\
   \|\nabla^2\rho\|_\infty\le r_0\big\}
\]
together with the $W_2$-displacement geodesics from $\mu_0$ to its members satisfies
Definition~\ref{def:reg}: every member has density in $[c/2,2C]$, the reference-to-member optimal maps
$\mathrm{id}+\nabla\rho$ are smooth diffeomorphisms, the family is star-geodesically closed around
$\mu_0$, and it is locally transport-rich there. In particular the packing of
Appendix~\ref{app:reduction} lies in $\mathcal F_{r_0}$, so Definition~\ref{def:reg} is nonvacuous and
the lower bound is not an artifact of an empty class.
\end{proposition}
\begin{proof}
This is Lemma~\ref{lem:bump} applied to the bump family. For $\|\nabla^2\rho\|_\infty\le r_0$ small,
the potential $\tfrac12|x|^2+\theta\rho$ is uniformly convex for every $\theta\in[0,1]$, so
$\mathrm{id}+\theta\nabla\rho$ is a Brenier diffeomorphism and the pushforward density
$1/\det(I+\theta\nabla^2\rho)$ lies in $[c/2,2C]$. Since $\rho\in C_c^\infty(\Omega)$ is supported in
the interior, $\mathrm{id}+\theta\nabla\rho$ equals the identity near $\partial\Omega$, so it maps
$\Omega$ diffeomorphically onto itself. The displacement geodesic from $\mu_0$ to a member
$(\mathrm{id}+\nabla\rho)_\#\mu_0$ is exactly $\theta\mapsto(\mathrm{id}+\theta\nabla\rho)_\#\mu_0$,
and since $\theta\rho$ obeys the same Hessian bound it stays in $\mathcal F_{r_0}$ for all
$\theta\in[0,1]$; this is simultaneously the star-geodesic closure around $\mu_0$ and the local
transport-richness required by Definition~\ref{def:reg}. We make no claim about the geodesic between
two arbitrary members: the Brenier map $(\mathrm{id}+\nabla\rho_2)\circ(\mathrm{id}+\nabla\rho_1)^{-1}$
between them is in general not a gradient perturbation of $\mu_0$, and the construction does not need
it --- Lemma~\ref{lem:reach} and Appendix~\ref{app:reduction} use only reference-to-endpoint
geodesics. Constancy of $\mu_0$ is used only by the separation estimates of
Appendix~\ref{app:reduction}, not here.
\end{proof}

\noindent The reduction needs a \emph{local} statistical-richness bound: for the reference $\mu_0$,
constants $c,r_0>0$ with
$\inf_{\hat\nu}\sup_{\nu:\,W_2(\nu,\mu_0)\le r}\E_\nu\,W_2(\hat\nu,\nu)\ge c\,\min(r,m^{-\gamma_d})$
for every $0<r\le r_0$ and sample size $m$, $\gamma_d=\min(1/d,1/2)$. A \emph{global} minimax rate
need not transfer to every shrinking ball, so rather than assume this we \emph{derive} it from
regularity --- it is exactly Proposition~\ref{prop:reduction}. Theorem~\ref{thm:unified} is then a
reduction to this local lower bound. This local, moving-target bound is distinct from the static shape
floor of Theorem~\ref{thm:stat}(A): there a fixed separated subfamily is observed directly with all
$M$ samples, whereas here the moving target confines the usable temporal pool to the in-window
snapshots, coupling temporal smoothness and spatial estimation into a single rate.

\begin{lemma}[Reachability under the smoothness budget]\label{lem:reach}
Assume the geometric regularity of Definition~\ref{def:reg}, and that the constant-speed
$W_2$-geodesic from $\mu_0$ to $\nu$ is regular with a well-defined velocity field $V$ along which
the covariant-derivative chain rule holds. Then for every $\nu$ with
$\delta:=W_2(\mu_0,\nu)\le \varepsilon H^{k+1}/(k+1)!$ there is a curve in $\Ck$ that equals
$\mu_0$ for $t\le t_n-H$ and reaches $\mu_{t_n}=\nu$.
\end{lemma}

\begin{theorem}[Unified lower bound via temporal-to-spatial reduction]\label{thm:unified}
For a regular problem (Definition~\ref{def:reg}) with $\|\nabla_t^k v_t\|\le\varepsilon$,
observed through $n$ snapshots of $N$ samples each ($M=Nn$ total), the minimax prediction risk
satisfies, for every $k\ge0$,
\[
   \inf_{\hat\nu}\sup_{\mu_\bullet\in\Ck}\E\,W_2(\hat\nu,\mu_{t_n+h})\ \gtrsim\
   \max\!\Big\{\tfrac{\varepsilon h^{k+1}}{(k+1)!}\,,\
   \sup_{0<H\le L}\min\!\Big(c_k\,\varepsilon H^{k+1},\,c\,(M^{\mathrm{eff}}_{H,k})^{-\gamma_d}\Big)\Big\},
\]
where $\gamma_d=\min(1/d,1/2)$ and
$M^{\mathrm{eff}}_{H,k}=N\sum_{i:\,t_i\in[t_n-H,t_n]}q_k\big(\tfrac{t_i-t_n+H}{H}\big)^2$ is the
design-weighted in-window information ($q_k$ the smoothstep schedule of the construction,
Appendix~\ref{app:reduction}; $0\le q_k\le1$). The first term is the exact-past floor of
Theorem~\ref{thm:lecam}, the second the temporal-to-spatial reduction (Lemma~\ref{lem:reach}, proved
in Appendix~\ref{app:reduction}). The bound is \emph{design-dependent}: $M^{\mathrm{eff}}_{H,k}$ is the
information actually carried by the snapshots inside the bandwidth, so the supremum is over feasible
windows with no continuity of $H$ assumed. Since $q_k\le1$, $M^{\mathrm{eff}}_{H,k}\le M_H:=Nn_H$, so
the cruder bound with $M_H$ in place of $M^{\mathrm{eff}}_{H,k}$ also holds.
\end{theorem}

\begin{corollary}[Closed-form rate; dense versus resolution-limited regimes]\label{cor:denserate}
For the equispaced design ($n_H\asymp1+H/\Delta$, and $M^{\mathrm{eff}}_{H,k}\asymp_k M_H\asymp MH/L$
for $H\ge\Delta$ by the Riemann sum $\tfrac1{n_H}\sum_i q_k(s_i)^2\to c_{q,k}\in(0,1)$), let
$H_\#=\big((L/M)^{\gamma_d}/\varepsilon\big)^{1/(k+1+\gamma_d)}$ be the unconstrained maximizer of the
inner $\min$. \emph{(i) Dense regime.} If $H_\#\ge\Delta$ --- the statistically optimal bandwidth
resolves at least one snapshot beyond the endpoint --- then, for fixed $k$,
\[
   \inf_{\hat\nu}\sup_{\mu_\bullet\in\Ck}\E\,W_2(\hat\nu,\mu_{t_n+h})\ \gtrsim_k\
   \varepsilon\,(h+H_\#)^{k+1}
   \qquad(\text{using }\max\{a^{k+1},b^{k+1}\}\asymp_k(a+b)^{k+1}),
\]
whose statistics-dominated branch ($h\lesssim H_\#$) has $M$-exponent $\gamma_d(k+1)/(k+1+\gamma_d)$:
the location rate $M^{-(k+1)/(2k+3)}$ of Theorem~\ref{thm:sharp} for $d\le2$, and
$M^{-(k+1)/(d(k+1)+1)}$ for $d\ge3$ (stated at the power-law level at $d=2$; the critical logarithmic
factor is not optimized by the single-scale packing of Appendix~\ref{app:reduction}).
\emph{(ii) Resolution-limited regime.} If instead $H_\#<\Delta$, the feasible windows obey $H\ge\Delta$
(or contain only the endpoint, where $M^{\mathrm{eff}}_{H,k}\asymp N$), so the inner $\sup$ is of the
same order as its value at the smallest resolved window $H\in[\Delta,2\Delta)$, namely
$\gtrsim_k\max\{\varepsilon h^{k+1},\ \min(c_k\,\varepsilon\Delta^{k+1},\,c\,N^{-\gamma_d})\}$ ---
governed by the temporal resolution $\Delta$ and the per-snapshot count $N$, not by $M$. The
closed-form exponent in (i) is thus the dense-temporal-design rate; the matching upper bounds below
operate in the same regime.
\end{corollary}
\begin{proof}
In the dense regime $M^{\mathrm{eff}}_{H,k}\asymp_k M_H\asymp MH/L$ (equispaced Riemann sum), so up to
$k$-constants the inner objective is $\min(c_k\varepsilon H^{k+1},\,c\,(MH/L)^{-\gamma_d})$; the first
factor increases and the second decreases in $H$, so the continuous maximizer balances them,
$\varepsilon H^{k+1}\asymp(MH/L)^{-\gamma_d}$, giving
$H_\#=((L/M)^{\gamma_d}/\varepsilon)^{1/(k+1+\gamma_d)}$ and value $\asymp_k\varepsilon H_\#^{k+1}$.
If $H_\#\ge\Delta$ this maximizer is feasible; combining with the exact-past floor by
$\max\{a^{k+1},b^{k+1}\}\asymp_k(a+b)^{k+1}$ gives (i), and substituting $H_\#$ into
$\varepsilon H_\#^{k+1}$ yields the stated $M$-exponents. If $H_\#<\Delta$ the balancing bandwidth is
infeasible; for $H>H_\#$ the binding term is the resolution $c\,(M^{\mathrm{eff}}_{H,k})^{-\gamma_d}$
(decreasing in $H$), while at the smallest resolved widths $H\asymp\Delta$ one has
$M^{\mathrm{eff}}_{H,k}\asymp N$, so the supremum is of the same order as its value at $H\asymp\Delta$,
namely $\min(c_k\varepsilon\Delta^{k+1},\,cN^{-\gamma_d})$, giving (ii).
\end{proof}

\begin{theorem}[Matching upper bound at $k=0$]\label{thm:upper0}
For $k=0$ on a regular problem, the pooled (persistence) estimator --- the empirical
distribution of all $M_H=N n_H$ samples in a window $[t_n-H,t_n]$ --- satisfies
\[
   \E\,W_2(\hat\nu,\mu_{t_n+h})\ \lesssim\ \varepsilon\,(h+H)\ +\ M_H^{-\gamma_d}.
\]
In the dense regime ($H_\#\ge\Delta$, so $M_H\asymp MH/L$) optimizing $H$ matches the lower bound of
Theorem~\ref{thm:unified}; hence the $k=0$ unified rate $M^{-1/(d+1)}$ ($d\ge3$; $M^{-1/3}$ for
$d\le2$) is sharp.
\end{theorem}

\begin{conjecture}[General-$k$ upper bound]\label{conj:upper}
For $k\ge1$ on a regular problem (equispaced dense design), a degree-$k$ temporal local-polynomial forecaster on the
tangent bundle (geodesic/barycentric regression~\cite{fletcher} of the snapshots with sample splitting) attains
bias $\lesssim\varepsilon(h+H)^{k+1}$ (the order-$k$ Otto--Taylor remainder,
Proposition~\ref{prop:sharp}) and variance $\lesssim(MH/L)^{-\gamma_d}$, hence meets the lower
bound of Theorem~\ref{thm:unified} and the unified exponent $M^{-(k+1)/(d(k+1)+1)}$. The
construction and constant are established here only in the location channel
(Theorem~\ref{thm:sharp}, all $k$) and end-to-end at $k=0$ (Theorem~\ref{thm:upper0},
Figure~\ref{fig:unified}); Appendix~\ref{app:forecaster} gives the explicit estimator and reduces
this conjecture to a curvature-stability estimate~(C) and an optimal-transport map-estimation
rate~(S), both unconditional at $k=0$ and on flat submodels.
\end{conjecture}

\begin{remark}
The exponent degrades from the static shape rate $M^{-\gamma_d}$ (Theorem~\ref{thm:stat}A, recovered
as $k\to\infty$: a frozen target permits unlimited pooling) because a moving target limits
temporal pooling; higher smoothness $k$ recovers more of it, and in $d\le2$ the spatial channel
is already rate-$M^{-1/2}$, adding nothing beyond the location channel. The proven envelope:
lower bound for all $k$ (Theorem~\ref{thm:unified}), sharp at $k=0$
(Theorem~\ref{thm:upper0}); for $k\ge1$ Appendix~\ref{app:covariant} gives a covariant
(development-based) forecaster and shows it attains the rate \emph{conditionally} on two estimates
--- a comparison-geometry bias bound and an optimal-transport map-estimation rate
(Proposition~\ref{prop:covariant}). The \emph{unconditional} $k\ge1$ upper bound remains open
(Conjecture~\ref{conj:upper}).
Figure~\ref{fig:unified} plots the exponent and the $M^{-\gamma_d}$ ingredient.
\end{remark}

\begin{openproblem}[Unconditional general-$k$ upper bound]\label{op:upper}
Exhibit an estimator $\widehat\nu$ and a finite constant $c_k$ such that, for every regular problem
(Definition~\ref{def:reg}) with $\|\nabla_t^k v\|\le\varepsilon$ and every $k\ge1$,
$\E\,W_2(\widehat\nu,\mu_{t_n+h})\le c_k\,\varepsilon\,(h+H_\#)^{k+1}$ at the window-optimal $H_\#$,
thereby matching the lower bound of Theorem~\ref{thm:unified} \emph{without} Assumptions~(C),(S). By
Proposition~\ref{prop:covariant} it suffices to establish, on the regular class: \emph{(C)} a uniform
sectional-curvature upper bound $\bar\kappa<\infty$ together with a valid operator-valued Rauch
comparison in $(\PP,W_2)$, yielding a curvature-free anti-development bias; and \emph{(S)} a pooled
Brenier-map estimation rate
$\E\|\widehat U-\E\,\widehat U\|_{L^2(\bar\mu)}^2\lesssim M_H^{-2\gamma_d}$ at the empirical-measure
exponent $\gamma_d=\min(1/d,1/2)$. Both reduce to established facts at $k=0$ and on flat submodels
(Appendix~\ref{app:covariant}); the open content is their validity at the infinite-dimensional,
positively-curved general-$k$ level.
\end{openproblem}

\section{Numerical illustration}\label{sec:numerics}
The theory is illustrated numerically in Appendix~\ref{app:numerics}: horizon exponents $h^{k+1}$
on a flat translation family and a curved (Bures--Wasserstein) Gaussian path; the unified rate
$M^{-(k+1)/(d(k+1)+1)}$ under the window-optimal budget; robustness to the temperature smoothing
window (Appendix~\ref{app:smooth}); and two real series --- near-stationary equity returns versus a
smoothly drifting seasonal temperature --- at opposite ends of the drift/noise spectrum
(Section~\ref{sec:realdata}). The experiments confirm the proven $k=0$ rate and are consistent with
the conditional general-$k$ prediction; the contribution of this paper is theoretical and no claim
rests on them.
\section{Discussion}

We have mapped the forecastability of a slowly-varying curve in $\PP$ into two regimes: an
exact-past extrapolation floor set purely by temporal smoothness, dimension-free and of order
$h^{k+1}$, and a finite-sample statistical floor governed by the spatial cost of estimating a
measure. Their interaction is the paper's main object: because a moving target caps temporal
pooling, the static empirical-$W_2$ rate $M^{-\gamma_d}$ is unattainable, and the forecast risk obeys
a unified lower bound with $M$-exponent $\gamma_d(k+1)/(k+1+\gamma_d)$. The three remarks below
record what each floor certifies; we then state what is sharp, what the data decide, and what
remains open.

\begin{remark}[Worst vs.\ average case]
The floors scale as $h^{k+1}$ (worst case) and $h^{k+1/2}$ (rms, average case): a sup-bound
lets an adversary push consistently, a random derivative cancels. Complementary, not
matching; both identify the order-$k$ extrapolator as optimal, respectively in worst-case scaling and in Bayes risk within the Gaussian translation model.
\end{remark}

\begin{remark}[Adiabatic hierarchy]
$k=0,1,2$ give persistence ($\varepsilon h$), geodesic ($\varepsilon h^2/2$), spline
($\varepsilon h^3/6$): controlled derivatives $=$ forecastable horizon exponent.
\end{remark}

\begin{remark}[What is sharp]
The worst-case constant is exact on the flat translation submodel (Theorem~\ref{thm:lecam}), the
average-case Bayes risk is exact within its Gaussian model (Theorem~\ref{thm:bayes}), and the
finite-sample statistical exponents (Theorems~\ref{thm:stat},~\ref{thm:sharp}) are rate-sharp in
the regimes stated. Over $\PP$ the unified \emph{lower} bound
(Theorem~\ref{thm:unified}) is rigorous for all $k$ on the regular class, recovering the classical
empirical-$W_2$ minimax exponent (Fournier--Guillin; Niles-Weed--Berthet) via the explicit packing of
Appendix~\ref{app:reduction}; the matching upper
bound is proved at $k=0$ (Theorem~\ref{thm:upper0}), and for $k\ge1$ a covariant forecaster
(Appendix~\ref{app:covariant}) attains it \emph{conditionally} on a comparison-geometry bias bound
and an optimal-transport map-estimation rate (Proposition~\ref{prop:covariant}), the unconditional
$k\ge1$ case remaining open (Conjecture~\ref{conj:upper}); numerically the curse $\gamma_d$ is
tracked to $d=6$ by two
independent OT solvers (on the predicted ordering; the higher-$d$ fits are pre-asymptotic) and the
unified exponent is reproduced by the measured-curse-plus-exact-bias
construction (Figure~\ref{fig:unified}); the endpoint-estimation experiment ($h=0$) lands on the
predicted band for $d=2$ and remains pre-asymptotic for $d=3$. Open, for the $k\ge1$ upper bound:
(i) well-posedness of the Cartan development and a uniform sectional-curvature bound on the regular
class in $(\PP,W_2)$, on which the bias estimate~(C) rests; (ii) a transport-map (not merely
distribution) estimation rate for the drift-corrected pooled estimator, estimate~(S); and the
curvature correction beyond $k\le2$ (Proposition~\ref{prop:sharp}).
\end{remark}

\paragraph{The effective extrapolation order is data-dependent.} Which forecaster is optimal is
determined by the regularity actually present, not fixed a priori. The two real series of
Section~\ref{sec:realdata} bracket this: on the near-stationary S\&P cross-sections degree-$0$
persistence is best, while on the strongly-drifting temperature field the horizon slope grows with
$k$ and the optimal pooling bandwidth is interior ($H^*=3$ days), the observable slope rising with the
drift-to-noise ratio. In both, the moving forecast floor sits well above the finite-sample noise
reference. Crucially, the calibrated bandwidth predicts the held-out optimum
(Figure~\ref{fig:heldout}), so the bias--variance trade-off is a genuine prediction, not a post-hoc
fit.

\paragraph{Limitations and outlook.} The unified upper bound is established end-to-end only at
$k=0$ and in the location channel for all $k$; for $k\ge1$ a covariant forecaster attains it only
\emph{conditionally} on the curvature-stability and optimal-transport map-estimation estimates
isolated in Appendices~\ref{app:forecaster}--\ref{app:covariant} (given which,
Proposition~\ref{prop:covariant} matches the lower bound), and the curvature correction itself is
controlled only for $k\le2$ (Proposition~\ref{prop:sharp}); the unconditional $k\ge1$ upper bound
remains Conjecture~\ref{conj:upper}. Under a finite memory budget on the past, these floors become
the high-rate limit of a rate--distortion curve, developed separately. Extending the empirical study
to deseasonalized residuals, downstream tasks, and longer horizons is left to future work.

\appendix

\section{Proofs}\label{app:proofs}
This appendix collects the proofs of the results stated in the main text, in order of appearance.

\begin{proof}[Proof of Lemma~\ref{lem:embed}]
$(\tau_x,\tau_y)_\#\rho$ has cost $|x-y|^2$, so $W_2\le|x-y|$; the translation
$z\mapsto z+(y-x)$ is the Brenier map between translates, attaining it. Displacement
interpolation of two translates is a translate, so geodesics stay in the image
(totally geodesic); the induced metric is Euclidean (flat).
\end{proof}

\begin{proof}[Proof of Lemma~\ref{lem:mean}]
For any coupling $\pi$, $|\mean\alpha-\mean\beta|=|\E_\pi[X-Y]|\le(\E_\pi|X-Y|^2)^{1/2}$;
infimize over $\pi$.
\end{proof}

\begin{proof}[Proof of Theorem~\ref{thm:lecam}]
Two curves $\mu^b_t=(\tau_{x_b(t)})_\#\rho$ with $x_b(t)=0$ for $t\le t_n$ and
$x_b(t)=(-1)^b\frac{\varepsilon}{(k+1)!}(t-t_n)^{k+1}e$ for $t>t_n$. By
Lemma~\ref{lem:embed} the directions are flat, so $\|\nabla_t^k v^b_t\|=|x_b^{(k+1)}|=\varepsilon$
on $(t_n,\infty)$, hence $\mu^b_\bullet\in\Ck$ (the jump at $t_n$ is admissible under
the ess-sup bound, or mollify; Lemma~\ref{lem:moll}). The curves agree on $[0,t_n]$,
so the forecaster is fixed, while $W_2(\mu^0_{t_n+h},\mu^1_{t_n+h})=2\varepsilon
h^{k+1}/(k+1)!$; the triangle inequality gives the bound.
\end{proof}

\begin{proof}[Proof of Theorem~\ref{thm:bayes}]
With $\hat a:=\mean(\hat\nu)-\mean(\rho)$, Lemma~\ref{lem:mean} gives
$W_2^2(\hat\nu,\mu_{t_n+h})\ge|\hat a - x(t_n+h)|^2$. The state $(x,\dot x,\dots,x^{(k)})$
is Markov, so the past fixes the $k$-jet at $t_n$; Taylor with integral remainder gives
$x(t_n+h)=\sum_{j\le k}\frac{x^{(j)}(t_n)}{j!}h^j+\frac1{k!}\int_0^h(h-s)^k\xi(t_n+s)ds$,
a conditional mean plus an independent residual of covariance
$\frac{\sigma^2 I_d}{(k!)^2}\frac{h^{2k+1}}{2k+1}$, of trace as claimed.
\end{proof}

\begin{proof}[Proof of Theorem~\ref{thm:stat}]
\emph{(A)} Restrict to static curves $\mu_t\equiv\mu\in\Ck$: data are $M$ i.i.d.\ draws
from $\mu$ and $\mu_{t_n+h}=\mu$, so forecasting is $W_2$ density estimation, minimax
$\asymp M^{-\gamma_d}$ ($\gamma_d=\min(1/d,1/2)$, up to a possible critical-dimension logarithmic
correction at $d=2$) by Fano packing (Niles-Weed--Berthet; Singh--P\'oczos; upper bound
Fournier--Guillin).
\emph{(B)} Two-point along $e$: $x_0\equiv0$ vs $x_1=p+b$, $p$ degree-$k$,
$b(t)=\frac{\varepsilon}{(k+1)!}(t-t_n)_+^{k+1}$. Both in $\Ck$; $b\equiv0$ on the window,
so data differ only through $p$. With $\KL(\rho\|\rho(\cdot-\delta e))=\tfrac12 I_e\delta^2
+O(\delta^3)$, $\KL(P_0\|P_1)\le\frac{N I_e}{2}a^\top G a$; impose $a^\top G a\le\frac1{N I_e}$
($\KL\le\tfrac12$, $\TV\le\tfrac12$). Le~Cam + Lemma~\ref{lem:mean} give risk
$\ge\frac14\big[\,|w^\top a|+\frac{\varepsilon h^{k+1}}{(k+1)!}\,\big]$; maximizing $|w^\top a|$ over the
ellipsoid gives $\sup_{a^\top G a\le(N I_e)^{-1}}|w^\top a|=(N I_e)^{-1/2}(w^\top G^{-1}w)^{1/2}$
(Cauchy--Schwarz in the $G$-metric), the bump $b$ adding the extrapolation floor.
\end{proof}

\begin{proof}[Proof of Theorem~\ref{thm:sharp}]
\emph{Lower bound (optimized bump width).} Take $x_0\equiv0$, $x_1=\phi$, where on
$[t_n-H,t_n]$ the function $\phi$ spends its full budget $\|\phi^{(k+1)}\|_\infty\le\varepsilon$
to build a $k$-jet $\phi^{(j)}(t_n)\asymp\varepsilon H^{k+1-j}$, then continues by degree-$k$
Taylor for $t>t_n$ (so $\phi^{(k+1)}=0$ there and $\phi\in\Ck$). The future separation is
$\phi(t_n+h)=\sum_{j=0}^k\frac{\phi^{(j)}(t_n)}{j!}h^j\asymp_k\varepsilon(h+H)^{k+1}$ (the bump matches
the endpoint $k$-jet up to $k$-dependent constants; the precise $1/(k+1)!$ is the maximal-jet
normalization and is not needed). On the window $\phi$ is supported
on $[t_n-H,t_n]$ with $|\phi|\le c\,\varepsilon H^{k+1}$, so
$\sum_i\phi(t_i)^2/\sigma_N^2\asymp\frac{M\varepsilon^2}{\sigma_1^2 L}H^{2k+3}$; demanding
indistinguishability ($\le1$) gives $H\le H_*$. Le~Cam with Lemma~\ref{lem:mean} yields risk
$\gtrsim\phi(t_n+h)$, maximized at $H=H_*$.
\emph{Upper bound.} A degree-$k$ local polynomial of bandwidth $H$ has worst-case bias
$\asymp\varepsilon(h+H)^{k+1}$ and variance $\asymp\sigma_1^2 L/(MH)$ for $h\lesssim H$;
minimizing $\varepsilon^2 H^{2(k+1)}+\sigma_1^2 L/(MH)$ over $H$ gives $H\asymp H_*$ and a
matching error.
\end{proof}

\begin{proof}[Proof of Lemma~\ref{lem:reach}]
Let $(\gamma_u)_{u\in[0,1]}$ be the unit-time constant-speed $W_2$-geodesic from $\mu_0$ to
$\nu$; its velocity field $V$ has $\|V\|\equiv\delta$ and $\nabla_V V=0$. With the time profile
$\theta(t)=\big((t-(t_n-H))/H\big)^{k+1}$ on $[t_n-H,t_n]$ --- so $\theta^{(j)}(t_n-H)=0$ for
$j\le k$ and $\theta^{(k+1)}\equiv(k+1)!/H^{k+1}$ --- set $\mu_t=\gamma_{\theta(t)}$ on the window
and $\mu_t=\mu_0$ before it. Then $v_t=\theta'(t)V$, and since $\nabla_V^{\,j}V=0$ for all
$j\ge1$ only the scalar derivatives of $\theta$ survive: $\nabla_t^k v_t=\theta^{(k+1)}(t)\,V$,
hence $\|\nabla_t^k v_t\|=\theta^{(k+1)}\,\delta=(k+1)!\,\delta/H^{k+1}\le\varepsilon$. The
pre-window piece has $v\equiv0$, so the curve lies in $\Ck$.
\end{proof}

\begin{proof}[Proof of Theorem~\ref{thm:unified}]
Fix $H$ and restrict $\Ck$ to the sub-family of curves equal to a fixed $\mu_0$ for
$t\le t_n-H$ that, by Lemma~\ref{lem:reach}, reach an arbitrary $\nu$ in the ball
$\mathcal B_H:=\{\nu:W_2(\mu_0,\nu)\le\varepsilon H^{k+1}/(k+1)!\}$ at $t_n$. Only the
$n_H=\#\{i:t_i\in[t_n-H,t_n]\}$ snapshots inside the window depend on $\nu$, and their information is
the design-weighted $M^{\mathrm{eff}}_{H,k}=N\sum_i q_k(s_i)^2\le M_H=Nn_H$
($M^{\mathrm{eff}}_{H,k}\asymp_k M_H\asymp MH/L$ in the equispaced dense regime $H\ge\Delta$,
Corollary~\ref{cor:denserate}). These are drawn not from $\nu$ but from the intermediate laws along the geodesic
$\mu_0\to\nu$, so the empirical-$W_2$ minimax bound cannot be transferred to $\nu$ directly.
Appendix~\ref{app:reduction} (Proposition~\ref{prop:reduction}) closes this step: an explicit
transport-map packing of $\mathcal B_H$ together with a Fano bound on the full snapshot experiment
shows that the window experiment has effective KL sample size $\asymp_k M^{\mathrm{eff}}_{H,k}$,
yielding for each feasible bandwidth $H$ the local minimax lower bound
$\gtrsim_k\min(c_k\,\varepsilon H^{k+1},\,c\,(M^{\mathrm{eff}}_{H,k})^{-\gamma_d})$; taking the supremum over feasible $H$
gives the second term of the displayed bound, with no continuity or monotonicity of $H\mapsto M_H$
used. Appendix~\ref{app:reduction} freezes each packing path after
$t_n$, so over that subfamily $\mu_{t_n+h}^{(\omega)}=\nu_\omega$ exactly and forecasting reduces to
endpoint estimation with no $h$-dependence; the first term is the separate exact-past floor
$\varepsilon h^{k+1}/(k+1)!$ of Theorem~\ref{thm:lecam}. The closed-form optimization of this
supremum --- the crossing $H=H_\#$ and the resulting $\varepsilon(h+H_\#)^{k+1}$, via
$\max\{a^{k+1},b^{k+1}\}\asymp_k(a+b)^{k+1}$ --- is carried out for the equispaced dense design in
Corollary~\ref{cor:denserate}.
\end{proof}

\begin{lemma}[Non-i.i.d.\ empirical $W_2$]\label{lem:pool}
Let $X_1,\dots,X_m$ be independent, $X_j\sim P_j$, all supported on a compact $\Omega\subset\R^d$, and
$\bar P_m=\tfrac1m\sum_j P_j$. Then
$\E\,W_2\big(\tfrac1m\sum_j\delta_{X_j},\,\bar P_m\big)\le C_{\Omega,d}\,m^{-\gamma_d}$,
$\gamma_d=\min(1/d,1/2)$ (up to the critical $d=2$ logarithm).
\end{lemma}
\begin{proof}
The Fournier--Guillin / Weed--Bach dyadic argument bounds $W_2$ by a weighted sum of
$\E\,|\hat P(Q)-\bar P_m(Q)|$ over dyadic cubes $Q$, $\hat P=\tfrac1m\sum_j\delta_{X_j}$. Independence
gives $\mathrm{Var}\big(\hat P(Q)\big)=\tfrac1{m^2}\sum_j P_j(Q)\big(1-P_j(Q)\big)\le \bar P_m(Q)/m$,
the same per-cube control as the i.i.d.\ case; identical distribution is never used, only independence
and bounded support. Summing over scales reproduces the i.i.d.\ rate.
\end{proof}

\begin{proof}[Proof of Theorem~\ref{thm:upper0}]
We bound a population bias and an empirical fluctuation separately.
\emph{Step 1 (population mixture bias).} The speed bound gives
$W_2(\mu_t,\mu_{t_n+h})\le\varepsilon(h+H)$ for all $t\in[t_n-H,t_n]$. Gluing the optimal couplings
of each $(\mu_{t_i},\mu_{t_n+h})$ shows the population pooled mixture
$\bar\mu_H=\sum_i\lambda_i\mu_{t_i}$ obeys
$W_2^2(\bar\mu_H,\mu_{t_n+h})\le\sum_i\lambda_i W_2^2(\mu_{t_i},\mu_{t_n+h})
\le\max_i W_2^2(\mu_{t_i},\mu_{t_n+h})\le(\varepsilon(h+H))^2$, i.e.\
$W_2(\bar\mu_H,\mu_{t_n+h})\le\varepsilon(h+H)$.
\emph{Step 2 (empirical fluctuation).} The pooled empirical measure $\hat\mu_H$ is built from the
$M_H=Nn_H$ window samples --- independent but not identically distributed, $N$ from each snapshot
$\mu_{t_i}$. By Lemma~\ref{lem:pool} (the dyadic empirical-$W_2$ argument needs only independence and
bounded support, not identical distribution) it concentrates on its mean law
$\bar\mu_H=\sum_i\lambda_i\mu_{t_i}$ at the empirical-$W_2$ rate,
$\E\,W_2(\hat\mu_H,\bar\mu_H)\lesssim M_H^{-\gamma_d}$. This is the fixed-$N$-per-snapshot design;
it is not identical to i.i.d.\ mixture sampling, where the per-snapshot counts would themselves be
random.
\emph{Step 3 (triangle inequality).} Hence
$\E\,W_2(\hat\mu_H,\mu_{t_n+h})\le W_2(\bar\mu_H,\mu_{t_n+h})+\E\,W_2(\hat\mu_H,\bar\mu_H)
\lesssim\varepsilon(h+H)+M_H^{-\gamma_d}$, as in the theorem statement for arbitrary design. In the
equispaced dense regime ($H\ge\Delta$), $M_H\asymp MH/L$, and optimizing $H$ recovers the $k=0$ rate
of Corollary~\ref{cor:denserate}.
Figure~\ref{fig:unified} is consistent with both terms and the optimized exponent.
\end{proof}

\section{The temporal--spatial reduction: a window-experiment Fano bound}\label{app:reduction}
This appendix proves the local minimax lower bound invoked in the proof of
Theorem~\ref{thm:unified}. The point it settles is that the window snapshots are drawn not from the
endpoint $\nu$ but from the intermediate laws $\mu_s^{\nu}$ along the geodesic $\mu_0\to\nu$;
reachability of $\nu$ (Lemma~\ref{lem:reach}) does not by itself make the window experiment
equivalent to direct sampling from $\nu$. We construct an explicit packing of the reachable ball and
bound the Kullback--Leibler (KL) divergence of the \emph{full} snapshot experiment, so that Fano's
inequality applies. The spatial packing follows the localized transport perturbation of Wasserstein
minimax lower bounds (Niles-Weed--Berthet; Weed--Bach); the smoothstep temporal embedding and the
full-window KL accumulation are the new ingredients.

\paragraph{Window experiment.} Take $\Omega=[0,1]^d$ and let $\mu_0\equiv1$ be the \emph{uniform}
(constant) reference density (Definition~\ref{def:reg}); only this single hard reference is needed,
and its constancy --- not merely a two-sided bound $0<c\le\mu_0\le C$ --- is what the density and
separation estimates below require. Parametrise the window by
$s\in[0,1]$, $s=(t-t_n+H)/H$, with a smoothstep schedule $q_k(s)$ in place of the bare $s^{k+1}$ of
Lemma~\ref{lem:reach}: $q_k(0)=0$, $q_k(1)=1$, $q_k^{(j)}(0)=q_k^{(j)}(1)=0$ for $1\le j\le k$, and
$\|q_k^{(k+1)}\|_\infty\le C_k$, so the path is $C^k$-flat at \emph{both} ends (flat at $s=0$ to glue
with the past constant curve, flat at $s=1$ so it can be frozen at $\nu_\omega$ afterwards). A concrete
choice is the regularized incomplete-beta profile
$q_k(s)=\int_0^s u^k(1-u)^k\,\mathrm du\,\big/\!\int_0^1 u^k(1-u)^k\,\mathrm du$, with
$C_k=\|q_k^{(k+1)}\|_\infty$ growing in $k$; all $\gtrsim_k$ constants below are $k$-dependent. There are
$n_H=\#\{i:t_i\in[t_n-H,t_n]\}$ in-window snapshots at $s_i$, each sampled $N$ times, $M_H:=Nn_H$ (for the equispaced design $n_H\asymp nH/L$ and $M_H\asymp MH/L$ once $H\ge\Delta$, Corollary~\ref{cor:denserate}).

\paragraph{Transport-map packing.} Partition $\Omega$ into $m^d$ subcubes of side $1/m$ with centres
$x_c$. Fix a smooth $\Phi$ compactly supported in the open unit cube $(0,1)^d$ with $\int\Phi=0$, and for
$\omega\in\{\pm1\}^{m^d}$ set
\[
  \rho_\omega(x)=\frac{a}{m}\sum_c \omega_c\,\Phi\big(m(x-x_c)\big),\qquad
  \nu_\omega=(\nabla\psi_\omega)_\#\mu_0,\quad \psi_\omega(x)=\tfrac12|x|^2+\rho_\omega(x).
\]
For $am\le c_0$ (a small constant) $\psi_\omega$ is convex, so $\nabla\psi_\omega=\mathrm{id}+\nabla\rho_\omega$
is a Brenier map and $\mu_s^{(\omega)}=(\mathrm{id}+q_k(s)\nabla\rho_\omega)_\#\mu_0$ traces the
$W_2$-geodesic from $\mu_0$ to $\nu_\omega$ on a $C^k$-flat schedule (Lemma~\ref{lem:bump}). The
displacement field $\nabla\rho_\omega=a\sum_c\omega_c\nabla\Phi(m(x-x_c))$ has amplitude $\asymp a$ on
each cell of Lebesgue volume $m^{-d}$ (equal to its $\mu_0$-mass, $\mu_0\equiv1$), so with constants
depending only on $\Phi,d$,
\[
  W_2(\mu_0,\nu_\omega)\asymp a,\qquad
  W_2(\nu_\omega,\nu_{\omega'})\asymp a\,\sqrt{d_H(\omega,\omega')/m^d},
\]
$d_H$ the Hamming distance; the upper bound is the common-source coupling and the matching lower
bound is the bi-Lipschitz estimate of Lemma~\ref{lem:sep}, with
$\|\nabla\rho_\omega-\nabla\rho_{\omega'}\|_{L^2(\mu_0)}^2\asymp a^2\,d_H/m^d$ over the $d_H$ differing
cells. By the Varshamov--Gilbert bound there is
$\mathcal W\subset\{\pm1\}^{m^d}$ with $|\mathcal W|\ge 2^{m^d/8}$ and pairwise $d_H\ge m^d/8$, hence
$W_2(\nu_\omega,\nu_{\omega'})\gtrsim a$ on $\mathcal W$.

\begin{lemma}[Transport separation for the bump packing: a two-sided $W_2$ bound]\label{lem:sep}
Let $\mu_0\equiv1$ be the uniform density on $\Omega=[0,1]^d$, and let $\rho_\omega,\rho_{\omega'}$ be
two potentials of the packing above with $am\le c_0$ for a small $c_0=c_0(\Phi,d)$. With
$u=\nabla\rho_\omega$, $v=\nabla\rho_{\omega'}$ and a constant $c'=c'(\Phi,d)>0$,
\[
  c'\,\|u-v\|_{L^2(\mu_0)}\ \le\ W_2\big((\mathrm{id}+u)_\#\mu_0,\,(\mathrm{id}+v)_\#\mu_0\big)\ \le\ \|u-v\|_{L^2(\mu_0)};
\]
in particular $W_2(\nu_\omega,\nu_{\omega'})\asymp a\,\sqrt{d_H(\omega,\omega')/m^d}$.
\end{lemma}
\noindent For a signed measure $\xi$ on $\Omega$ with $\xi(\Omega)=0$ we use the weighted homogeneous
norm $\|\xi\|_{\dot H^{-1}(\mu_0)}:=\sup\{\int_\Omega g\,\mathrm d\xi:\,g\in C^\infty(\Omega),\,\int_\Omega|\nabla g|^2\,\mathrm d\mu_0\le1\}$,
equal to $\|\nabla\Lambda\xi\|_{L^2(\mu_0)}$, where $\Lambda\xi$ solves the weighted Neumann problem
$-\nabla\!\cdot(\mu_0\nabla\Lambda\xi)=\xi$ on $\Omega$, $\partial_n\Lambda\xi|_{\partial\Omega}=0$;
since $\mu_0\equiv1$ it coincides with the unweighted $\dot H^{-1}(\mathrm dx)$ norm.
\begin{proof}
\emph{Upper bound.} The common-source coupling $x\mapsto\big((\mathrm{id}+u)(x),(\mathrm{id}+v)(x)\big)$
has cost $\int_\Omega|u-v|^2\,\mathrm d\mu_0$, so $W_2\le\|u-v\|_{L^2(\mu_0)}$ unconditionally.

\emph{Lower bound.} Let $p_w$ be the density of $\nu_w:=(\mathrm{id}+w)_\#\mu_0$. Since $\mu_0\equiv1$,
$p_w=1/\det\!\big(I+\nabla w\circ(\mathrm{id}+w)^{-1}\big)$, so $\|p_w-1\|_\infty\lesssim\|\nabla w\|_\infty\lesssim am$
for $am\le c_0$ (Lemma~\ref{lem:bump}(i)); in particular $p_u,p_v\in[1-C'am,1+C'am]\subset[c/2,2C]$.
Peyr\'e's non-asymptotic comparison \cite[Thm.~1]{peyre} gives
$W_2(\nu_u,\nu_v)\ge c''\,\|p_u-p_v\|_{\dot H^{-1}(\mathrm dx)}$ with $c''=\big(2(1+C'c_0)\big)^{-1/2}$,
and $\dot H^{-1}(\mathrm dx)=\dot H^{-1}(\mu_0)$ as $\mu_0\equiv1$.

We compare $p_v-p_u$ with its linearisation, \emph{using the Eulerian velocity}. Along
$w_\tau=(1-\tau)u+\tau v$ set $T_\tau=\mathrm{id}+w_\tau$; the pushforward curve
$\tau\mapsto\nu_{w_\tau}=(T_\tau)_\#\mu_0$ solves the continuity equation
$\partial_\tau p_\tau+\nabla\!\cdot(p_\tau\,b_\tau)=0$ whose Eulerian velocity is the Lagrangian
velocity $\partial_\tau w_\tau=v-u$ \emph{evaluated at the current configuration},
$b_\tau=(v-u)\circ T_\tau^{-1}$ (writing $v-u$ directly in the Eulerian variable would drop this
composition). Integrating in $\tau$ and decomposing
$p_{w_\tau}\,b_\tau=(v-u)+R_\tau$ with
$R_\tau=(p_{w_\tau}-1)(v-u)+p_{w_\tau}\big[(v-u)\circ T_\tau^{-1}-(v-u)\big]$,
\[
  (p_v-p_u)+\nabla\!\cdot\!\big(\mu_0(v-u)\big)=-\!\int_0^1\nabla\!\cdot R_\tau\,\mathrm d\tau.
\]
As $v-u=\nabla(\rho_{\omega'}-\rho_\omega)$ is a gradient, the Neumann solution of
$-\nabla\!\cdot(\mu_0\nabla\Lambda)=\nabla\!\cdot(\mu_0(v-u))$ is $\Lambda=-(\rho_{\omega'}-\rho_\omega)$,
whence $\|\nabla\!\cdot(\mu_0(v-u))\|_{\dot H^{-1}(\mu_0)}=\|v-u\|_{L^2(\mu_0)}$ \emph{exactly}.
Both pieces of $R_\tau$ are supported on the $d_H$ differing cells: since $\Phi$ is compactly
supported in the open unit cube, every $T_\tau$ is the identity near each cell boundary and maps each
cell onto itself (\emph{cell preservation}), so $T_\tau^{-1}(y)$ lies in the same cell as $y$ and
$v-u$ vanishes off the differing cells, where $|v-u|\asymp a$, $\|\nabla(v-u)\|_\infty\asymp am$, and
$\|w_\tau\|_\infty\lesssim a$. Hence, for any $g$ with $\int|\nabla g|^2\mathrm d\mu_0\le1$:
\emph{(a) density remainder} ---
$\big|\int g\,\nabla\!\cdot\!\big((p_{w_\tau}-1)(v-u)\big)\big|=\big|\int\nabla g\cdot(p_{w_\tau}-1)(v-u)\big|
\le\|p_{w_\tau}-1\|_\infty\|v-u\|_{L^2(\mu_0)}\lesssim am\,\|v-u\|_{L^2(\mu_0)}$;
\emph{(b) composition remainder} --- on each differing cell
$|(v-u)\circ T_\tau^{-1}-(v-u)|\le\|\nabla(v-u)\|_\infty\,|T_\tau^{-1}-\mathrm{id}|\lesssim(am)\,a$ while
$|v-u|\asymp a$, so $\|(v-u)\circ T_\tau^{-1}-(v-u)\|_{L^2(\mu_0)}\lesssim am\,\|v-u\|_{L^2(\mu_0)}$ and the
$\dot H^{-1}(\mu_0)$-norm of $\nabla\!\cdot\!\big(p_{w_\tau}[\cdots]\big)$ is
$\lesssim am\,\|v-u\|_{L^2(\mu_0)}$. Combining, $\|p_u-p_v\|_{\dot H^{-1}(\mu_0)}\ge(1-C''am)\|v-u\|_{L^2(\mu_0)}$,
so $W_2(\nu_u,\nu_v)\ge c'\,\|u-v\|_{L^2(\mu_0)}$ with $c'=c''\big(1-C''c_0\big)>0$ for $c_0$ small.
Both remainders use only $\|\nabla(u-v)\|_\infty\asymp am$ and the constancy of $\mu_0$ --- no
spatial regularity of the reference beyond constancy.

For the packing $u-v=\nabla(\rho_\omega-\rho_{\omega'})$ has amplitude $\asymp a$ on the $d_H$ differing
cells, each of Lebesgue volume $m^{-d}$ (equal to its $\mu_0$-mass since $\mu_0\equiv1$), so
$\|u-v\|_{L^2(\mu_0)}\asymp a\sqrt{d_H/m^d}$.
\end{proof}

\begin{lemma}[Bump regularity]\label{lem:bump}
There is $c_0=c_0(\Phi,c,C)>0$ such that for $am\le c_0$: \emph{(i)} each $\nu_\omega$ has density in
$[c/2,2C]$ and $\mathrm{id}+\theta\nabla\rho_\omega$ is a diffeomorphism of $\Omega$ for all
$\theta\in[0,1]$; \emph{(ii)} the displacement interpolation is the $W_2$-geodesic and lies in the
regular class; \emph{(iii)} along it $\|\nabla_t^k v_t\|\le C_k\,W_2(\mu_0,\nu_\omega)/H^{k+1}$ with
$C_k=\|q_k^{(k+1)}\|_\infty$, so the path lies in $\Ck$ whenever
$W_2(\mu_0,\nu_\omega)\le r_H:=\varepsilon H^{k+1}/C_k\asymp_k\varepsilon H^{k+1}$.
\end{lemma}
This is the standard small-$C^2$-perturbation argument for Brenier maps (cf.\ \cite{hutter,manole});
part \emph{(iii)} is the reparametrisation identity of Lemma~\ref{lem:reach} with the smoothstep
constant $C_k$.

\begin{lemma}[Hamming-localized density separation]\label{lem:kl}
For $am\le c_0$ and every $s\in[0,1]$,
\[
  \big\|\mu_s^{(\omega)}-\mu_s^{(\omega')}\big\|_{L^2(\Omega)}^2\ \le\ C\,q_k(s)^2\,a^2 m^2\,\frac{d_H(\omega,\omega')}{m^d},
\]
$C=C(\Phi,c,C)$.
\end{lemma}
\begin{proof}
Since $\Phi$ is compactly supported in the open unit cube, $\nabla\Phi$ vanishes in a neighbourhood
of each cell boundary; as $T^{(s)}_\omega:=\mathrm{id}+q_k(s)\nabla\rho_\omega$ is a diffeomorphism
equal to the identity there, it maps each cell onto itself (cell preservation). Hence the inverse
$x=(T^{(s)}_\omega)^{-1}(y)$ lies in the same cell as $y$, where both $\nabla\rho_\omega$ and
$\nabla^2\rho_\omega$ depend only on $\omega_c$. \emph{Because $\mu_0\equiv1$ is constant the
numerator of the pushforward density carries no spatial dependence}:
$p_s^{(\omega)}(y)=1/\det\!\big(I+q_k(s)\nabla^2\rho_\omega(x)\big)$ depends, on cell $c$, only on
$\omega_c$. In particular, for two codewords $\omega,\omega'$ the otherwise-delicate numerator
difference $\mu_0(x_\omega)-\mu_0(x_{\omega'})$ (which for a merely bounded reference would require
Lipschitz/Sobolev control) \emph{vanishes identically}, so $p_s^{(\omega)}-p_s^{(\omega')}$ is
supported on the $d_H$ differing cells and is controlled by the determinant alone. There
$\nabla^2\rho$ has amplitude $am$ (since $\nabla^2[\tfrac{a}{m}\Phi(m\cdot)]=am\,\Phi''(m\cdot)$), and
with $q_k(s)\,am\le c_0$ the determinant expands as $1+O(q_k(s)\,am)$, so
$|p_s^{(\omega)}-p_s^{(\omega')}|\lesssim q_k(s)\,am$ over Lebesgue volume $m^{-d}$ per cell (its
$\mu_0$-mass, $\mu_0\equiv1$). Only this upper bound is needed (the matching $W_2$ \emph{lower}
separation is Lemma~\ref{lem:sep}). Squaring and summing over the $d_H$ differing cells gives the
stated bound.
\end{proof}

\paragraph{KL of the snapshot experiment.} Write $\delta p_s^{(\omega)}=\mu_s^{(\omega)}-\mu_0$. A
change of variables gives $\delta p_s^{(\omega)}=-q_k(s)\,\nabla\!\cdot(\mu_0\nabla\rho_\omega)+R_s^{(\omega)}$
with $\|R_s^{(\omega)}\|_{L^2}\lesssim (am)^2$. Since all densities are $\ge c/2$ and $am\le c_0$, a
fixed constant $C$ (depending only on $c_0,\Phi,c$) bounds
\[
  D_{\mathrm{KL}}\big(\mu_s^{(\omega)}\,\big\|\,\mu_s^{(\omega')}\big)
  \le \frac{1}{c}\int_\Omega\big(\mu_s^{(\omega)}-\mu_s^{(\omega')}\big)^2
  \le C\,q_k(s)^2\,\mathcal I(\omega,\omega'),\qquad
  \mathcal I(\omega,\omega'):=a^2 m^2\,\frac{d_H(\omega,\omega')}{m^d}
\]
(the $L^2$ separation is Lemma~\ref{lem:kl}; a constant upper bound, not a leading-order equivalence,
is all Fano needs, so the $am=O(1)$ remainder is absorbed into $C$). Summing the $N$ draws at each of
the $n_H$ snapshots, the pairwise KL of the full snapshot experiment is
\[
  D_{\mathrm{tot}}(\omega,\omega'):=\sum_i N\,D_{\mathrm{KL}}\big(\mu_{s_i}^{(\omega)}\|\mu_{s_i}^{(\omega')}\big)
  \ \lesssim_k\ M^{\mathrm{eff}}_{H,k}\,\mathcal I(\omega,\omega'),
  \qquad
  M^{\mathrm{eff}}_{H,k}:=N\!\!\sum_{i:\,t_i\in[t_n-H,t_n]}\!\! q_k(s_i)^2,
\]
the \emph{design-weighted} effective sample size. Since $0\le q_k\le1$ we have
$M^{\mathrm{eff}}_{H,k}\le Nn_H=M_H$ for every design, and the endpoint snapshot alone ($s_i=1$,
$q_k(1)=1$) forces $M^{\mathrm{eff}}_{H,k}\ge N$. The smoothstep concentrates information near the
endpoint ($q_k(s_i)\approx0$ for small $s_i$), so a design clustered at the window start carries
strictly less information than its raw count $M_H$ suggests --- the shape-channel analogue of the
location-channel leverage of Theorem~\ref{thm:sharp}, with $\sum_i q_k(s_i)^2$ in place of
$\sum_i(t_i-t_n)^{2k}$. \emph{For the equispaced design} the Riemann sum gives
$\sum_i q_k(s_i)^2\approx n_H\!\int_0^1 q_k(s)^2\,\mathrm ds=n_H\,c_{q,k}$ with $c_{q,k}\in(0,1)$ a
$k$-constant, so $M^{\mathrm{eff}}_{H,k}\asymp_k M_H$; for a general design only the inequality
$M^{\mathrm{eff}}_{H,k}\le M_H$ is used.

\begin{proposition}[Local minimax lower bound over the reachable ball]\label{prop:reduction}
Under Definition~\ref{def:reg}, with the reachable ball $\mathcal B_H=\{\nu:W_2(\mu_0,\nu)\le r_H\}$,
$r_H\asymp_k\varepsilon H^{k+1}$, and the design-weighted effective sample size $M^{\mathrm{eff}}_{H,k}$
of the window experiment above,
\[
  \inf_{\hat\nu}\ \sup_{\mu_\bullet\in\Ck:\,\mu_{t_n}\in\mathcal B_H}\ \E\,W_2(\hat\nu,\mu_{t_n})
  \ \gtrsim_k\ \min\!\big(r_H,\ c\,(M^{\mathrm{eff}}_{H,k})^{-\gamma_d}\big),
  \qquad\gamma_d=\min(1/d,1/2).
\]
For the equispaced design $M^{\mathrm{eff}}_{H,k}\asymp_k M_H\asymp MH/L$ (for $H\ge\Delta$), recovering
the closed-form rate of Corollary~\ref{cor:denserate}.
\end{proposition}

\begin{proof}
Use the packing $\{\nu_\omega\}_{\omega\in\mathcal W}$ at the largest admissible amplitude $a$. Two
constraints bound $a$: the path must lie in $\Ck$, i.e.\ $a\lesssim r_H\asymp_k\varepsilon H^{k+1}$
(Lemma~\ref{lem:bump}(iii)), and the map must stay monotone, $am\le c_0$. Pairwise separation is
$\gtrsim a$ (Varshamov--Gilbert, $d_H\ge m^d/8$), and since $d_H\le m^d$ the total KL obeys
$D_{\mathrm{tot}}\lesssim_k M^{\mathrm{eff}}_{H,k}\,a^2m^2$. Fano's inequality (\cite{tsybakov}, Thm.~2.5)
gives a lower bound of order the separation once
$D_{\mathrm{tot}}+\log2\le\tfrac12\log|\mathcal W|\asymp m^d$, i.e.
\[
  M^{\mathrm{eff}}_{H,k}\,a^2m^2\ \lesssim_k\ m^d
  \quad\Longleftrightarrow\quad
  a^2\ \lesssim_k\ \frac{m^{d-2}}{M^{\mathrm{eff}}_{H,k}}.
\]
Maximising $a$ over $m$ subject to $am\le c_0$: for $d\ge3$ the binding choice is
$m\asymp(M^{\mathrm{eff}}_{H,k})^{1/d}$, giving $a\asymp(M^{\mathrm{eff}}_{H,k})^{-1/d}$; for $d\le2$ the
optimum is at $m\asymp1$, where the packing reduces to a \emph{fixed finite} (Varshamov--Gilbert)
hypothesis set and the bound is equivalently a Le~Cam/finite-Fano two-point-type argument, giving the
parametric $a\asymp(M^{\mathrm{eff}}_{H,k})^{-1/2}$ (the critical $d=2$ logarithmic factor is \emph{not}
produced by this single-scale packing and is not claimed here). Thus the largest separation a
$\Ck$-admissible packing supports is $\asymp_k\min(r_H,\,(M^{\mathrm{eff}}_{H,k})^{-\gamma_d})$, and the
minimax error is $\gtrsim_k$ this value.
\end{proof}

\paragraph{From the ball to the forecast.} Because the smoothstep $q_k$ is $C^k$-flat at $s=1$ (all
derivatives up to order $k$ vanish at $t_n$), the constant continuation $\mu_t^{(\omega)}=\nu_\omega$
for $t\ge t_n$ is a genuine $C^k$ extension --- no velocity jump --- and keeps the path in $\Ck$.
Hence $\mu_{t_n+h}^{(\omega)}=\nu_\omega$, so forecasting $\mu_{t_n+h}$ over this sub-family is exactly
estimating $\nu_\omega\in\mathcal B_H$, and Proposition~\ref{prop:reduction} lower-bounds the forecast
risk by the spatial term of Theorem~\ref{thm:unified}. Combined with the disjoint exact-past floor of
Theorem~\ref{thm:lecam} (location channel) and optimised over $H$, this yields the stated
$\gtrsim_k\varepsilon(h+H_\#)^{k+1}$. The packing lives in the class by the local transport-richness of
Definition~\ref{def:reg}; Proposition~\ref{prop:reduction} is thus the promised local
statistical-richness bound, derived from regularity rather than assumed.

\begin{remark}[The lower bound is unconditional; consolidated constants]\label{rem:lbconstants}
Unlike the matching upper bound for $k\ge1$ (Proposition~\ref{prop:covariant}, conditional on the
estimates (C),(S) of Appendix~\ref{app:covariant}), the lower bound proved here invokes no
unverified hypothesis beyond Definition~\ref{def:reg}. It uses only: the constant reference
$\mu_0\equiv1$; a single compactly supported bump $\Phi$; Peyr\'e's non-asymptotic
$W_2$--$\dot H^{-1}$ comparison \cite{peyre}; the Varshamov--Gilbert and Fano inequalities
\cite{tsybakov}; and the Fournier--Guillin/Weed--Bach empirical-$W_2$ rate \cite{fg,weedbach}. The
chain of constants is explicit and finite at each fixed $k$: $c_0=c_0(\Phi,d)$ (monotonicity,
Lemma~\ref{lem:bump}), $c''=(2(1+C'c_0))^{-1/2}$ (Peyr\'e), $c'=c''(1-C''c_0)>0$ (separation,
Lemma~\ref{lem:sep}), $C_k=\|q_k^{(k+1)}\|_\infty$ and $c_{q,k}=\int_0^1q_k(s)^2\,\mathrm ds\in(0,1)$
(smoothstep schedule); none vanishes or diverges at fixed $k$. Consequently the design-dependent
lower bound (with the weighted $M^{\mathrm{eff}}_{H,k}$) holds \emph{unconditionally} for every
$k\ge0$ and \emph{arbitrary} observation design; its closed-form exponent
$\gamma_d(k+1)/(k+1+\gamma_d)$ and the floor $\varepsilon(h+H_\#)^{k+1}$ additionally use the
equispaced dense design (Corollary~\ref{cor:denserate}, where $M^{\mathrm{eff}}_{H,k}\asymp_k M_H$).
With the $k=0$ matching upper bound (Theorem~\ref{thm:upper0}) the
characterization is tight at $k=0$; the general-$k$ upper bound is the sole remaining gap
(Conjecture~\ref{conj:upper}).
\end{remark}

\section{A degree-$k$ tangent-space forecaster behind Conjecture~\ref{conj:upper}}\label{app:forecaster}

This appendix makes Conjecture~\ref{conj:upper} concrete: we give an explicit degree-$k$
forecaster on $\PP$, decompose its error, prove the parts that are unconditional ---
recovering the rate at $k=0$ and on flat submodels --- and isolate what remains open as two
named estimates, \textup{(C)} and \textup{(S)}.

\subsection{Construction}
Write $s=t-t_n$, so the window is $s\in[-H,0]$ and the target is $s=h>0$. Fix a kernel $K$
supported on $[-1,0]$ and set $K_H(s)=K(s/H)$. The $N$ draws at each snapshot are split into two
folds $\mathcal D_0,\mathcal D_1$ (sample splitting).

\emph{Step 1 (base point).} From $\mathcal D_0$ form an estimate $\bar\mu$ of $\mu_{t_n}$ --- the
kernel-weighted Wasserstein barycenter of the windowed empirical snapshots, or simply the
empirical measure of the snapshot nearest $t_n$.

\emph{Step 2 (chart coordinates).} For each windowed snapshot $t_i$, using $\mathcal D_1$,
estimate the optimal transport map $\hat T_i$ from $\bar\mu$ to the empirical measure
$\hat\mu_{t_i}$, and set the \emph{log coordinate} $\hat U_i:=\hat T_i-\mathrm{id}\in
L^2(\bar\mu;\R^d)$. Its population version is $U(s_i):=\mathrm{Log}_{\bar\mu}\mu_{t_n+s_i}
=T^{\bar\mu\to\mu_{t_n+s_i}}-\mathrm{id}$.

\emph{Step 3 (degree-$k$ regression in the chart).} In the Hilbert space $L^2(\bar\mu;\R^d)$ solve
\[
   (\hat A_0,\dots,\hat A_k)=\arg\min_{A_j\in L^2(\bar\mu)}\ \sum_i K_H(s_i)\,
   \Big\|\hat U_i-\textstyle\sum_{j=0}^k A_j\,s_i^{\,j}\Big\|_{L^2(\bar\mu)}^2 .
\]
This separates over $\bar\mu$-a.e.\ $x$ into scalar degree-$k$ local polynomial regressions of
$\{\hat U_i(x)\}$ on $\{s_i\}$, so $\hat A_j(x)=\sum_i \omega_j(s_i)\,\hat U_i(x)$ with the usual
local-polynomial weights $\omega_j$ (the same for every $x$, depending only on $\{s_i\}$ and $K_H$).

\emph{Step 4 (extrapolate and lift).} Evaluate the fitted polynomial at $s=h$,
$\hat U_\star=\sum_{j=0}^k \hat A_j\,h^{j}$, and output
$\hat\nu=\mathrm{Exp}_{\bar\mu}(\hat U_\star)=(\mathrm{id}+\hat U_\star)_\#\bar\mu$.
At $k=0$ the regression returns the kernel-weighted average of the coordinates (a
transport-barycentric persistence); the mixture forecaster of Theorem~\ref{thm:upper0} is an
equally valid degree-$0$ rule and is the one analyzed there.

\subsection{Error decomposition}
Since the Brenier map from $\bar\mu$ to the nearby target exists on the regular class,
$\mathrm{Exp}_{\bar\mu}U(h)=\mu_{t_n+h}$. The exponential is $1$-Lipschitz from $L^2(\bar\mu)$ to
$(\PP,W_2)$ --- $W_2((\mathrm{id}+a)_\#\bar\mu,(\mathrm{id}+b)_\#\bar\mu)\le\|a-b\|_{L^2(\bar\mu)}$,
the two maps coupling the measures --- so
\[
   W_2(\hat\nu,\mu_{t_n+h})\ \le\ \|\hat U_\star-U(h)\|_{L^2(\bar\mu)}
   \ \le\ \underbrace{\big\|\E\hat U_\star-U(h)\big\|}_{\text{(B) in-chart bias}}
   \ +\ \underbrace{\big\|\hat U_\star-\E\hat U_\star\big\|}_{\text{(V) variance}},
\]
expectations over $\mathcal D_1$ given $\bar\mu$ (the folds are independent by Step~1).

\subsection{The in-chart bias (B): unconditional given chart smoothness}
If $s\mapsto U(s)\in C^{k+1}([-H,h];L^2(\bar\mu))$ with $\sup_s\|U^{(k+1)}(s)\|_{L^2(\bar\mu)}
\le\varepsilon'$, the standard local-polynomial remainder (Fan--Gijbels~\cite{fangijbels}), applied $\bar\mu$-pointwise
and integrated, gives
\[
   \big\|\E\hat U_\star-U(h)\big\|_{L^2(\bar\mu)}\ \le\ C_{k,K}\,\varepsilon'\,(h+H)^{k+1}.
\]
This is rigorous: it uses only the chart-curve smoothness $\varepsilon'$ and boundedness of the
weights $\omega_j$ for $h\lesssim H$ (a well-conditioned design Gram matrix, as in
Corollary~\ref{cor:leverage}).

\subsection{The two open estimates}
\paragraph{(C) Chart stability (curvature).} The in-chart derivative $\partial_s^{k+1}U$ differs
from the intrinsic covariant derivative $\nabla_t^k v$ by terms involving the curvature of $\PP$
contracted with lower-order velocities (Gigli's second-order calculus). On flat submodels ---
translation (Lemma~\ref{lem:embed}) and Gaussian/Bures families, where $\bar\mu$-geodesics are
affine --- $\mathrm{Log}_{\bar\mu}$ is an isometry on the window, so $\partial_s^{k+1}U=\nabla_t^k v$
and $\varepsilon'=\varepsilon$ exactly; Proposition~\ref{prop:sharp} shows the leading correction
otherwise enters only at order $h^{k+3}$ for $k\le2$.
\emph{Assumption (C): on the regular class $\sup_s\|U^{(k+1)}(s)\|_{L^2(\bar\mu)}\le c_1\varepsilon$.}

\paragraph{(S) Chart estimation rate.} The coordinate $\hat U_i$ requires estimating the Brenier
map $\bar\mu\to\mu_{t_i}$ from $N$ samples. On the regular class, plug-in and entropic map
estimators converge in $L^2(\bar\mu)$ at the empirical-measure scale (H"utter--Rigollet~\cite{hutter};
Manole et al.~\cite{manole}; Pooladian--Niles-Weed~\cite{pooladian}), and sample splitting makes $\bar\mu\perp\{\hat U_i\}$.
As $\hat U_\star=\sum_{i,j}\omega_j(s_i)h^j\hat U_i$ is a fixed bounded-weight linear combination
pooling $M_H\asymp MH/L$ samples,
\[
   \E\big\|\hat U_\star-\E\hat U_\star\big\|_{L^2(\bar\mu)}\ \lesssim\ (MH/L)^{-\gamma_d},
   \qquad \gamma_d=\min(1/d,1/2).
\]
\emph{Assumption (S): the displayed variance bound holds on the regular class.}

\subsection{Conclusion}
Under \textup{(C)} and \textup{(S)},
$\E\,W_2(\hat\nu,\mu_{t_n+h})\lesssim\varepsilon(h+H)^{k+1}+(MH/L)^{-\gamma_d}$, and optimizing
$H$ as in Theorem~\ref{thm:upper0} yields the unified rate of Theorem~\ref{thm:unified}; the
forecaster is then minimax-rate-optimal, which \emph{would} establish Conjecture~\ref{conj:upper}
(stated as the conditional Proposition~\ref{prop:covariant}). Both estimates
hold unconditionally at $k=0$ --- \textup{(C)} is the $1$-Lipschitz drift bound $W_2(\mu_t,
\mu_{t_n+h})\le\varepsilon(h+H)$ and \textup{(S)} the empirical-$W_2$ rate
(Theorem~\ref{thm:upper0}) --- and for all $k$ on flat/Gaussian submodels (chart isometric, maps
affine and estimable at $N^{-\gamma_d}$). The residual content of the conjecture is therefore
exactly \textup{(C)} for $k\ge1$ in the curved regime, a quantitative second-order Otto-calculus
estimate, and \textup{(S)}, an optimal-transport map-estimation rate matching the empirical-measure
exponent --- both of a kind studied in the literature, though not, to our knowledge, in the combined
window-regression form required here. Appendix~\ref{app:covariant} carries out that combination,
giving an explicit covariant (development-based) forecaster and reducing the $k\ge1$ bound to
\textup{(C)} and \textup{(S)} stated as assumptions; establishing them unconditionally on the regular
class remains open.

\section{A covariant forecaster for the general-$k$ upper bound}\label{app:covariant}
Appendix~\ref{app:forecaster} reduced Conjecture~\ref{conj:upper} to the chart-stability
estimate~(C) and the map-estimation rate~(S). This appendix gives an explicit \emph{covariant}
forecaster and reduces the $k\ge1$ matching upper bound to two clean estimates, which we state as
\textbf{Assumptions~(C) and~(S)} below; we verify them at $k=0$ and on flat/Gaussian submodels, and
the resulting bound (Proposition~\ref{prop:covariant}) is therefore \emph{conditional} on these
assumptions for $k\ge1$. We do not claim to establish them unconditionally on the regular class: a
caveat is in order, since $(\PP,W_2)$ is not a finite-dimensional smooth manifold, and the global
existence, regularity, and curvature bounds underlying the development calculus we use are themselves
nontrivial (see \cite{ags,gigli} for the available second-order structure and
Section~\ref{ssec:covscope} for what remains open). The construction nonetheless makes the obstruction
to~(C) precise --- it is genuine for the \emph{naive} in-chart forecaster --- and identifies the
covariant remedy. We keep the notation of Appendix~\ref{app:forecaster}: $\bar\mu$ the in-window
reference, $U=\mathrm{Log}_{\bar\mu}(\cdot)$ the chart, $v=\dot\gamma$, and $\nabla_t$ the
Levi-Civita covariant derivative on $(\PP,W_2)$ (Otto calculus; Gigli's second-order
structure~\cite{gigli}). Independently of the conditional rate, two ingredients here are exact and
self-contained and may be read on their own: the development jet identity~\eqref{eq:devjet} and the
order-$(k{+}1)$ curvature cancellation (Remark~\ref{rem:naivefails}), which pinpoints \emph{why} a
naive in-chart extrapolator fails and what a covariant one must cancel.

\subsection{The covariant forecaster}\label{ssec:covdef}
Let $\widetilde\gamma:[0,h]\to T_{\bar\mu}\PP$ be the Cartan development of $\gamma$, i.e.\
$\widetilde\gamma\,'(s)=P_{0\leftarrow s}\dot\gamma(s)$ with $P_{0\leftarrow s}$ parallel transport
along $\gamma$. Using $\tfrac{d}{ds}\big(P_{0\leftarrow s}X(s)\big)=P_{0\leftarrow s}\nabla_t X(s)$ and
$P_{0\leftarrow0}=\mathrm{id}$ gives, for all $j\ge1$,
\begin{equation}\label{eq:devjet}
   \widetilde\gamma^{(j)}(0)=\nabla_t^{\,j-1}v\big|_{0}.
\end{equation}
The \emph{degree-$k$ covariant forecaster} extrapolates the development by its degree-$k$ Taylor
polynomial and maps back through the anti-development:
\[
   \widehat\nu \;=\; \mathrm{antidev}\!\Big(s\mapsto\textstyle\sum_{j=1}^{k}\tfrac{s^{j}}{j!}\,
   \widehat{\nabla_t^{\,j-1}v}\Big)\Big|_{s=h},
\]
the covariant derivatives $\widehat{\nabla_t^{\,j-1}v}$ being estimated by the tangent-bundle
regression of Appendix~\ref{app:forecaster}. Expanding the anti-development in the chart yields the
explicit form
$\widehat U_{\mathrm{cov}}(h)=\sum_{j=1}^{k}\tfrac{h^{j}}{j!}\widehat{\nabla_t^{\,j-1}v}
+\sum_{i\ge k+1}h^{i}c_i$, whose corrections $c_i$ are built from the estimated jet and the
curvature of $\PP$ at $\bar\mu$~\cite{gigli}; the leading correction is at order $h^{k+1}$. For
$k=3$ it equals $-\tfrac{h^{4}}{24}R(v,\nabla_t v)v$, the order-$4$ coefficient of $U\circ\gamma$
minus $\nabla_t^{3}v$.

\subsection{Bias: the estimate (C)}\label{ssec:covbias}
The bias splits into the development Taylor remainder --- exact and curvature-free in
$T_{\bar\mu}\PP$ --- and the Lipschitz stability of the anti-development. We isolate the geometric
content as an explicit assumption and indicate the candidate argument, which we do \emph{not} claim
to be rigorous in the infinite-dimensional setting (Section~\ref{ssec:covscope}).

\begin{assumption}[Comparison-geometry estimate (C)]\label{ass:C}
On the regular class the Cartan development and anti-development of admissible curves exist and are
unique, the chart $\mathrm{Log}_{\bar\mu}$ and the Brenier maps from the barycenter $\bar\mu$ to the
snapshots are $C^k$ (Caffarelli regularity --- the pairwise/chart smoothness deliberately excluded
from Definition~\ref{def:reg}), the relevant variation family is differentiable, the sectional
curvature of $(\PP,W_2)$ along the spanned $2$-planes lies in a fixed $[0,\bar\kappa]$ with
$\bar\kappa<\infty$, and the anti-development is $C_{\mathrm{geo}}$-Lipschitz from development curves
(sup-norm) to $W_2$ with $C_{\mathrm{geo}}=1+O(\bar\kappa\|v\|^2h^2)$ bounded for every $h$.
\end{assumption}

\noindent\textbf{Candidate argument (non-rigorous).} Linearise the anti-development along
$\sigma_\lambda=\widetilde\gamma+\lambda\big(T_k[\widetilde\gamma]-\widetilde\gamma\big)$; the
variation field $J_\lambda$ formally solves a forced Jacobi equation
$\nabla_s^{2}J_\lambda+R(J_\lambda,\dot c)\dot c=u_\lambda\ddot\delta+\mathcal H_\lambda$ with a
holonomy term $\mathcal H_\lambda$ from the frame variation. For $\sec\ge0$ (global on $\PP$ for a
Euclidean base~\cite{ags,villani}) the homogeneous Jacobi Green's function is bounded by the flat
$(s-\tau)$ \emph{with no conjugate-point restriction} --- conjugate points obstruct the two-point
boundary problem, not this forced initial-value problem --- suggesting a contraction
$C_{\mathrm{geo}}=1-O(\bar\kappa\|v\|^2h^2)$, with the holonomy contributing a higher-order
$O(\bar\kappa\|v\|\,\|\nabla_t v\|h^3)$ that is dominated in the slow-variation regime
$\|\nabla_t v\|h\lesssim\|v\|$. Turning this into a theorem requires the well-posedness and uniform
curvature bound of Assumption~\ref{ass:C} and an operator-valued Rauch comparison in $\PP$, none of
which we establish; hence (C) is an assumption, not a lemma.

\begin{proposition}[Bias under (C)]\label{prop:covbias}
Let $k\ge1$ and grant Assumption~\ref{ass:C}. If
$\sup_{[0,h]}\|\nabla_t^{k}v\|_{L^2(\bar\mu)}\le\varepsilon$, then
$W_2\big(\widehat\nu,\mu_{t_n+h}\big)\le C_{\mathrm{geo}}\,\varepsilon\,h^{k+1}/(k+1)!$.
\end{proposition}
\begin{proof}
By \eqref{eq:devjet} the development of $\gamma$ and its degree-$k$ Taylor agree to first order at
$0$, and their difference $\delta$ has $\|\ddot\delta(s)\|\le\varepsilon s^{k-1}/(k-1)!$; the Lagrange
remainder gives $\int_0^h(h-\tau)\|\ddot\delta\|\,d\tau=\varepsilon h^{k+1}/(k+1)!$ in
$T_{\bar\mu}\PP$, with no curvature term (parallel transport is isometric). The
$C_{\mathrm{geo}}$-Lipschitz anti-development of Assumption~\ref{ass:C} turns this into the stated
$W_2$ bias.
\end{proof}
\begin{remark}[The naive forecaster fails (C)]\label{rem:naivefails}
The order-$(k{+}1)$ chart derivative of $U\circ\gamma$ equals $\nabla_t^{k}v$ plus curvature
corrections each carrying at least one covariant acceleration $\nabla_t^{j}v$, $j\ge1$: along any
$W_2$-geodesic the chart curve is the straight line $s\mapsto sv$ and all $\nabla_t^{j}v$ vanish, so
a \emph{pure-velocity} correction would contradict straightness and hence cannot occur; mixed
corrections do. For $k=3$ the correction is $-R(v,\nabla_t v)v$, of size
$\asymp\|R\|\,\|v\|^{2}\|\nabla_t v\|$, which is $O(1)$ rather than $O(\varepsilon)$ unless the
entire jet is small. Thus the naive in-chart forecaster of Appendix~\ref{app:forecaster} does
\emph{not} achieve the curvature-free bias of estimate~(C) for $k\ge3$, whereas the covariant
forecaster cancels the correction by construction --- this is precisely why a covariant extrapolation
is needed.
\end{remark}

\subsection{Variance: the estimate (S)}\label{ssec:covvar}
\begin{lemma}[Optimality of the empirical exponent on the class]\label{lem:gammaopt}
The packing of Appendix~\ref{app:reduction} may be taken with bounded potential Hessian
($am\le c_0$, hence inside the regular class) while the perturbed density has Hessian of order
$m^{2}$. The densities therefore do not lie in a fixed H\"older ball as $m\to\infty$, so density
smoothness cannot be exploited and the empirical-$W_2$ exponent $\gamma_d=\min(1/d,1/2)$
(Fournier--Guillin~\cite{fg}; Weed--Bach~\cite{weedbach}; cf.\ the smooth regime of~\cite{nwb}) is
minimax-optimal on the class.
\end{lemma}
\begin{assumption}[Map-estimation rate (S)]\label{ass:S}
On the regular class, the chart estimate of the degree-$k$ covariant forecaster --- the Brenier
(transport) map obtained by drift-correcting the in-window samples to the forecast time via the
estimated jet and pooling the resulting $M_H\asymp MH/L$ samples --- estimates the drift-removed chart
value in $L^2(\bar\mu)$ at the rate
$\E\|\widehat U(0)-\E\,\widehat U(0)\|_{L^2(\bar\mu)}^2\lesssim(M_H)^{-2\gamma_d}$.
\end{assumption}

\noindent Estimate~(S) is a \emph{transport-map} rate, not the distribution rate, and the two are not
interchangeable: $W_2(\widehat\mu,\mu)\asymp M^{-\gamma_d}$ does not by itself give the same
$L^2(\bar\mu)$ rate for the Brenier map. Map-estimation bounds of this order are available under
additional regularity --- smooth maps, strongly convex potentials (H\"utter--Rigollet~\cite{hutter};
Manole et al.~\cite{manole}) --- but not for an arbitrary bounded-density class, which is why we state
(S) as an assumption rather than derive it from Lemma~\ref{lem:gammaopt}. By Lemma~\ref{lem:gammaopt}
the exponent in~(S), \emph{if} attained, cannot be improved; achievability itself is the content of
the assumption, since a lower bound does not certify an estimator.

\noindent\textbf{Candidate argument (non-rigorous) and the design factor.} Granting~(S), the
forecast $\widehat U_{\mathrm{cov}}(h)=\sum_{j=0}^k\tfrac{h^j}{j!}\widehat{U^{(j)}(0)}
=q_k(h/H)^\top\widehat\beta$ is a fixed linear functional of the degree-$k$ local-polynomial
coefficients $\widehat\beta$; with the in-window snapshots ($N$ samples each, $\asymp nH/L$ of them,
$M_H=MH/L$ pooled) uniform on the rescaled window, the weighted Gram matrix tends to the Hilbert
moment matrix $\widetilde M_k=(1/(i{+}j{+}1))_{0\le i,j\le k}$, so the pooled value-variance
$(M_H)^{-2\gamma_d}$ of~(S) propagates to the forecast through the \emph{exact} design quadratic form
\[
   \E\big\|\widehat U_{\mathrm{cov}}(h)-\E\,\widehat U_{\mathrm{cov}}(h)\big\|_{L^2(\bar\mu)}^{2}
   \ \lesssim\ (MH/L)^{-2\gamma_d}\,\Phi_k(h/H),
\]
where $\Phi_k(r)=q_k(r)^\top\widetilde M_k^{-1}q_k(r)$ with $q_k(r)=(1,-r,\dots,(-r)^k)$, and whose two
governing diagonal entries are exact and classical,
$\Phi_k(0)=[\widetilde M_k^{-1}]_{00}=(k+1)^2$ and
$[\widetilde M_k^{-1}]_{kk}=(2k+1)\binom{2k}{k}^2\sim\tfrac{2}{\pi}16^k$
(Fan--Gijbels~\cite{fangijbels}); sample-splitting keeps $\bar\mu\perp\widehat\beta$. For $h\lesssim H$
the prefactor is the polynomial $(k+1)^2$, so the variance is $\asymp(MH/L)^{-2\gamma_d}$ uniformly in
$k$; the exponential constant governs only far extrapolation $h\gg H$. The single non-rigorous step is
(S) itself; the choice to \emph{pool} rather than average snapshots is forced by resolution --- for
$d\ge3$ the empirical-$W_2$ error is a common, non-cancelling deficit, so averaging the $\asymp nH/L$
snapshots stays floored at the single-snapshot $N^{-\gamma_d}$ (a naive $N^{-\gamma_d}(nH/L)^{-1/2}$
would undercut Theorem~\ref{thm:unified}), while pooling reaches the $M_H$-sample resolution; for
$d\le2$ the two coincide.

\subsection{Conditional general-$k$ characterization}\label{ssec:covthm}
\begin{proposition}[Conditional matching upper bound for $k\ge1$, equispaced dense design]\label{prop:covariant}
Let $k\ge1$ and grant Assumptions~\ref{ass:C} and~\ref{ass:S}. On a regular problem under the
equispaced dense design ($H_\#\ge\Delta$, so $M_H\asymp MH/L$) the degree-$k$
covariant forecaster attains
\[
   \E\,W_2\big(\widehat\nu,\mu_{t_n+h}\big)\;\lesssim\;\varepsilon\,(h+H)^{k+1}+(MH/L)^{-\gamma_d},
\]
and optimizing $H=H_\#=\big((L/M)^{\gamma_d}/\varepsilon\big)^{1/(k+1+\gamma_d)}$ gives, for
$h\lesssim H_\#$, the $M$-exponent $\gamma_d(k+1)/(k+1+\gamma_d)$ of Theorem~\ref{thm:unified}. Thus,
\emph{conditionally on (C) and (S)}, the upper bound matches the lower bound for all $k$ in the
window-limited regime $h\lesssim H_\#$. The unconditional statement (without (C),(S)) is
Conjecture~\ref{conj:upper}.
\end{proposition}
\begin{proof}
Add the bias of Proposition~\ref{prop:covbias} ($\le C_{\mathrm{geo}}\varepsilon h^{k+1}/(k+1)!$ under
(C), and $\varepsilon(h+H)^{k+1}$ once the window fit over $[-H,h]$ is included) and the variance
under~(S) via the design factor of Section~\ref{ssec:covvar} ($\asymp(MH/L)^{-2\gamma_d}$ for
$h\lesssim H$, since $\Phi_k(h/H)\asymp(k+1)^2$ there), then optimize over $H$ as in
Theorem~\ref{thm:upper0}; the exponent is unimprovable by Theorem~\ref{thm:unified} and
Lemma~\ref{lem:gammaopt}.
\end{proof}

\subsection{Necessity of the regularity}\label{ssec:necessity}
\begin{proposition}[Sufficiency and rate-tightness of the regularity]\label{prop:necessity}
The two regularity layers play distinct, non-gratuitous roles. \emph{(Lower bound, spatial.)} By Lemma~\ref{lem:gammaopt}
the packing of Appendix~\ref{app:reduction} already saturates the empirical-$W_2$ exponent $\gamma_d$
while remaining inside the reference-star class of Definition~\ref{def:reg}, so strengthening the
smoothness of the densities cannot improve the spatial rate --- the \emph{spatial exponent is tight on
this class}. \emph{(Upper bound, chart.)} The conditional forecaster instead needs the
chart regularity of Assumption~\ref{ass:C}: Caffarelli regularity --- convex support and density
bounded away from $0$ and $\infty$ --- is what makes the Brenier maps from the barycenter, hence the
chart $\mathrm{Log}_{\bar\mu}$ and the development~\eqref{eq:devjet}, well defined and $C^k$, so that
the jet $\nabla_t^{j}v$ exists and is estimable; dropping it removes the very object the forecaster
extrapolates. Thus the spatial exponent is tight and the chart assumption is necessary for the
construction to be \emph{defined}; the remaining structural conditions (star-geodesic closure,
reference-map smoothness) are sufficient ingredients used by the proofs rather than individually
shown necessary.
\end{proposition}

\subsection{Scope}\label{ssec:covscope}
We separate what is established from what is assumed. \emph{Exact:} the development identity
\eqref{eq:devjet}, the curvature-cancellation of Remark~\ref{rem:naivefails} (the order-$4$
coefficient $-R(v,\nabla_t v)v$), and the design constants $\Phi_k(0)=(k+1)^2$,
$[\widetilde M_k^{-1}]_{kk}=(2k+1)\binom{2k}{k}^2$. \emph{Assumed} (and hence the source of the
``conditional'' in Proposition~\ref{prop:covariant}): Assumption~\ref{ass:C}, comprising
(a)~well-posedness of the Cartan development/anti-development, the variation calculus, and the
Caffarelli chart regularity (smooth Brenier maps from the barycenter, excluded from
Definition~\ref{def:reg}) in
$(\PP,W_2)$ --- which, unlike a finite-dimensional manifold, does not follow automatically and rests
on the second-order theory of~\cite{ags,gigli}; (b)~a \emph{uniform} sectional-curvature upper bound
$\bar\kappa<\infty$ over the regular class, not implied by Definition~\ref{def:reg} (Wasserstein
curvature depends on the densities and potential derivatives, and need not be uniformly bounded); and
(c)~validity of the operator-valued, variable-curvature Rauch/Jacobi comparison in this setting, of
which the scalar Green-kernel computation in the candidate argument of Section~\ref{ssec:covbias} is
only the model case. And
Assumption~\ref{ass:S}, a \emph{transport-map} (not distribution) estimation rate for the pooled,
drift-corrected estimator, known only under extra map regularity~\cite{hutter,manole} and not for an
arbitrary bounded-density class. \emph{Caveats:} the contraction $C_{\mathrm{geo}}\le1$ holds only in
the slow-variation regime $\|\nabla_t v\|h\lesssim\|v\|$ (the holonomy term has positive sign); the
operating regime is the window-limited $h\lesssim H_\#$; the $d=2$ logarithmic correction
(Appendix~\ref{app:d2}) is left at the power-law level. Establishing (C) and (S) unconditionally on
the regular class --- in particular a uniform curvature bound and a map-estimation rate at the
empirical-measure exponent --- is the content of Conjecture~\ref{conj:upper} and is left open.

\section{Numerical verification}\label{app:numerics}

\paragraph{Exponents survive curvature (Figure~\ref{fig:exp}).} We compare degree-$k$
extrapolation on a flat translation family (right; $W_2$ is Euclidean on the mean) against a
\emph{curved} path of zero-mean Gaussians with rotating eigenvectors (left; the covariances do
not commute, so the path is genuinely curved in the Bures--Wasserstein manifold), scored by the
closed-form Bures distance. Both give fitted log--log slopes $\approx k+1$ --- curved
$0.96,2.03,3.01$ and flat $0.98,1.99,2.99$ for $k=0,1,2$ --- consistent with the prediction that,
in this tested finite-dimensional submodel, curvature changes the leading \emph{constant} but not the
local horizon exponent $h^{k+1}$ (the curved/flat ratio runs $0.50,0.93,1.90$ across $k$); cf.\
Proposition~\ref{prop:sharp}/Lemma~\ref{lem:curv}.

\begin{figure}[tbp]
\centering
\includegraphics[width=\textwidth]{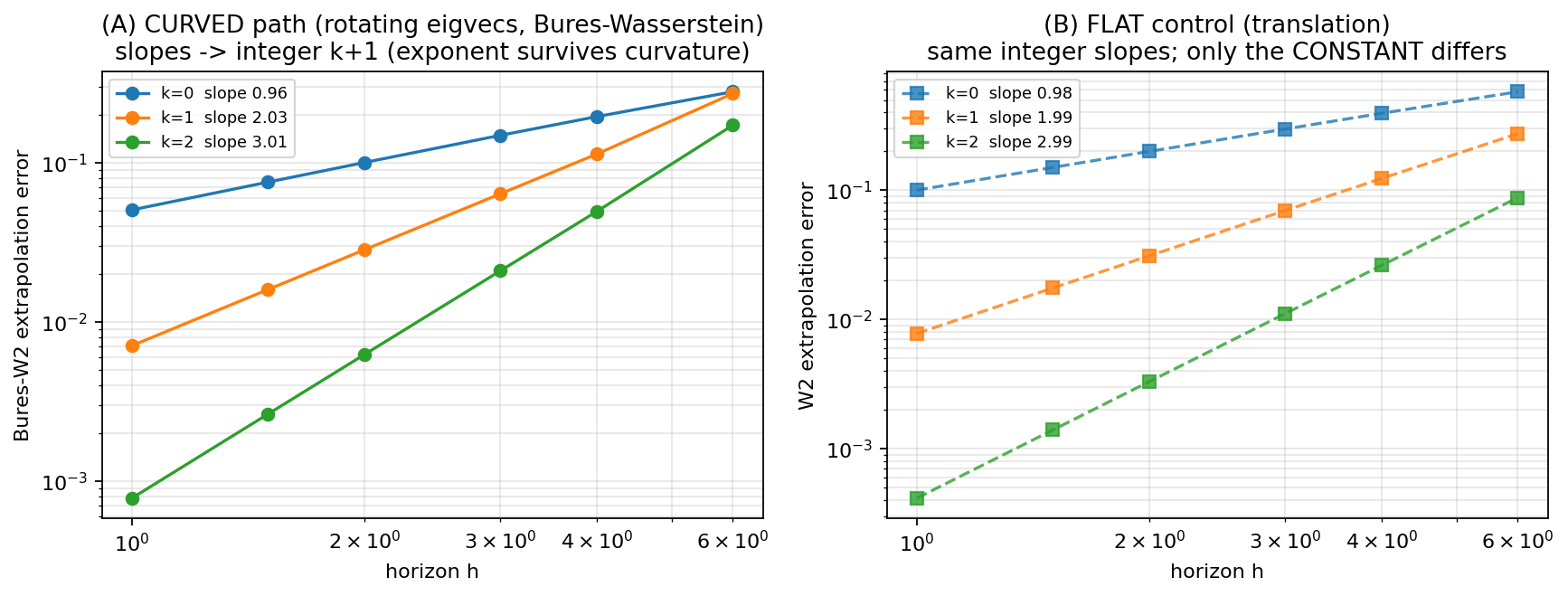}
\caption{Horizon exponent survives curvature (Proposition~\ref{prop:sharp}). \textbf{(Left)} a
curved path of zero-mean Gaussians with rotating eigenvectors (Bures--Wasserstein), degree-$k$
Taylor extrapolation, closed-form Bures error: slopes $\approx1,2,3$. \textbf{(Right)} a flat
translation control: slopes $\approx1,2,3$. The integer exponent is identical on both; only the
constant differs.}
\label{fig:exp}
\end{figure}

\paragraph{$(N,h)$ phase diagram (Figure~\ref{fig:phase}).} This verifies
Theorem~\ref{thm:stat}(B) and Corollary~\ref{cor:leverage} directly, using the
lower-bound construction itself: $\rho=N(0,1)$, $n=8$ ($L=7$), $k=1$, window truth a
degree-$k$ trend and the future carrying the invisible bump $b$. The order-$k$
least-squares extrapolant then has, exactly, bias equal to the extrapolation floor
$\varepsilon h^{k+1}/(k+1)!$ (the future deviation is information-theoretically
unobservable) and variance equal to the leverage $\tfrac1N w^\top G^{-1}w$. Panel~A maps
the RMS $\sqrt{\E W_2^2}=\sqrt{\text{floor}^2+\text{var}}$ over $(N,h)$; the white phase
boundary $\varepsilon h^{k+1}/(k+1)!=\sqrt{v}$ separates the extrapolation-limited regime
(error set by $h$, independent of $N$: the dimension-free floor) from the
statistics-limited regime ($N^{-1/2}$ leverage). Panels~B--C confirm the limiting
scalings ($N^{-1/2}\!\to$ bias plateau; slope $k\!\to\!k{+}1$ in $h$); Monte Carlo
(markers) matches the analytic risk to within $1.3\%$. The fitted large-$h$ slope of
$\sqrt{w^\top G^{-1}w}$ is $0.97\approx k$ with effective scale $6.4\approx L$, confirming
that the leverage is governed by the window $L$, not the spacing.

\begin{figure}[tbp]
\centering
\includegraphics[width=\textwidth]{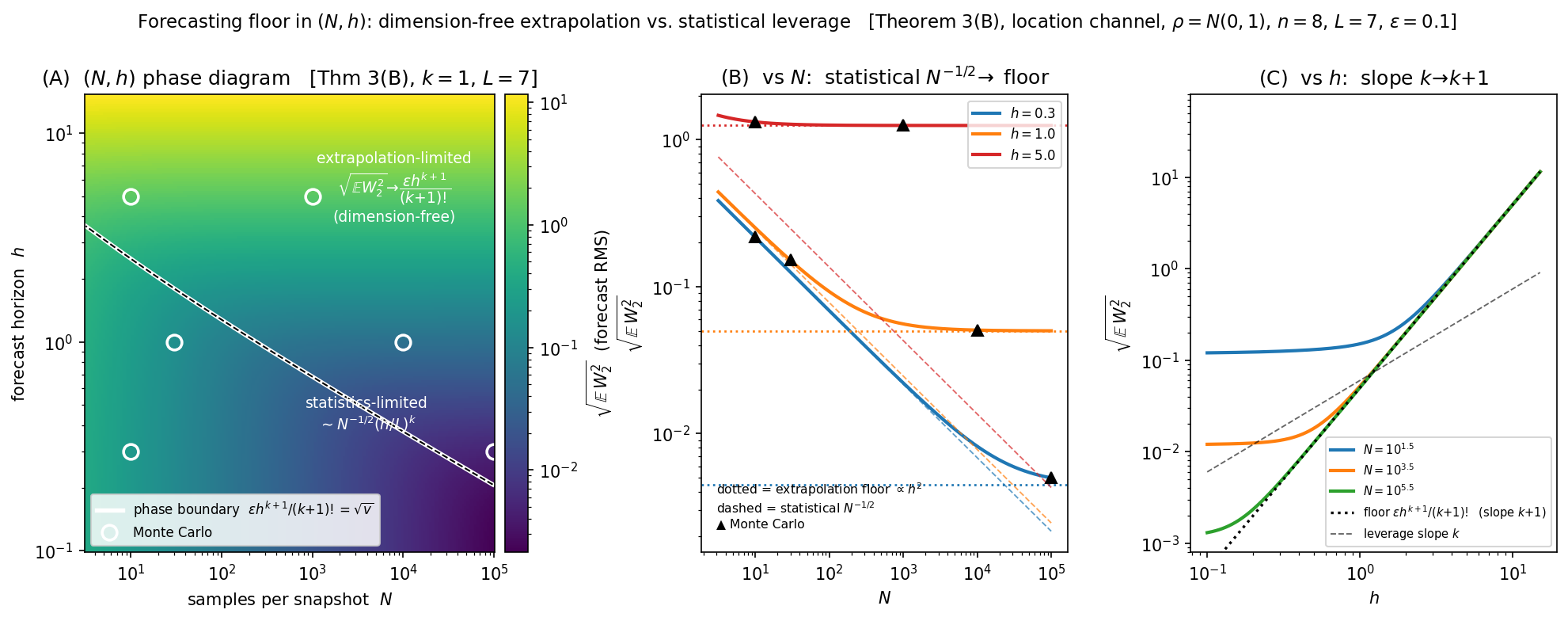}
\caption{$(N,h)$ phase diagram for Theorem~\ref{thm:stat}(B) (location channel,
$\rho=N(0,1)$, $n=8$, $L=7$, $\varepsilon=0.1$, $k=1$). \textbf{(A)} forecast RMS over
$(N,h)$; the white curve is the phase boundary $\varepsilon h^{k+1}/(k+1)!=\sqrt{v}$
separating the dimension-free extrapolation-limited regime (upper/right) from the
statistics-limited regime $\sim N^{-1/2}(h/L)^k$ (lower/left); circles are Monte Carlo.
\textbf{(B)} RMS vs.\ $N$ at fixed $h$: statistical $N^{-1/2}$ decay (dashed) settling onto
the extrapolation floor (dotted, $\propto h^2$); triangles are Monte Carlo. \textbf{(C)}
RMS vs.\ $h$ at fixed $N$: slope $k$ (leverage) crossing to slope $k{+}1$ (floor).}
\label{fig:phase}
\end{figure}

\paragraph{Sharp extrapolation rate (Figure~\ref{fig:sharp}).} Using a dense design
($n=600$) and a small horizon ($h=0.02$) to remain in the $h$-independent deep-statistics
regime $H_*\gg h$, the optimized-bandwidth local-polynomial forecaster has error decaying as
$M^{-(k+1)/(2k+3)}$. The fitted exponents $0.314,0.390,0.421,0.438$ for $k=0,1,2,3$ track the
theoretical $0.333,0.400,0.429,0.444$ (Figure~\ref{fig:sharp}B; the small undershoot is the
expected pre-asymptotic bias) and stand well clear of the loose parametric $1/2$. Monte
Carlo matches the analytic bias--variance to within $0.5\%$.

\begin{figure}[tbp]
\centering
\includegraphics[width=\textwidth]{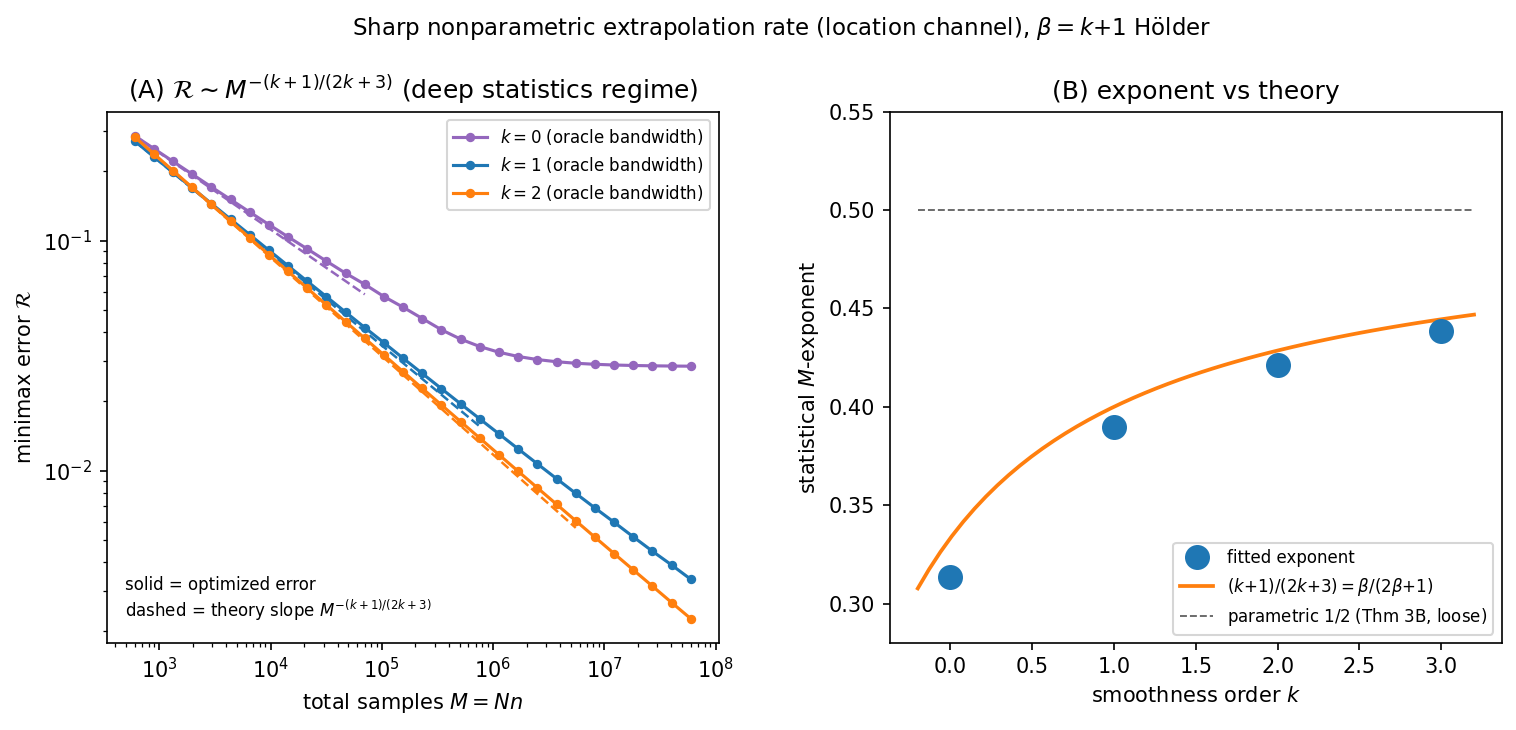}
\caption{Sharp nonparametric extrapolation rate (Theorem~\ref{thm:sharp}), location channel,
$\beta=k+1$ H\"older. \textbf{(A)} optimized-bandwidth forecast error vs.\ $M$ (solid) with
theoretical slope $M^{-(k+1)/(2k+3)}$ (dashed); the floor is off-scale at this small $h$.
\textbf{(B)} fitted statistical $M$-exponent vs.\ $k$ against $(k+1)/(2k+3)=\beta/(2\beta+1)$
(solid) and the loose parametric $1/2$ (dashed).}
\label{fig:sharp}
\end{figure}

\paragraph{Unified rate over $\PP$ (Figure~\ref{fig:unified}).} Three pieces, on isotropic
Gaussians drifting in $\R^d$. \textbf{(1) The spatial curse.} The empirical-$W_2$ fluctuation
$\E\,W_2(\hat\mu_M,\hat\mu_M')$ between two $M$-sample clouds --- a two-sample proxy for the
estimation risk $\E\,W_2(\hat\mu_M,\mu)$, which shares its exponent --- decays at the predicted
$M^{-\min(1/d,1/2)}$: a debiased Sinkhorn divergence~\cite{cuturi,feydy} on the GPU gives fitted exponents
$0.39,0.31,0.23,0.19,0.17$ for $d=2,\dots,6$ (theory $0.50,0.33,0.25,0.20,0.17$), and an
\emph{independent} exact network-simplex solver~\cite{peyrecuturi} reproduces $0.39,0.31,0.24$
for $d=2,3,4$ --- the two optimal-transport solvers agree to within $0.01$, so the measured curse
is not an entropic-regularization artifact. The $d=2$ undershoot ($0.39$ vs.\ $0.50$) is the
boundary log-correction. The $d=2$ undershoot ($0.39$ vs.\ $0.50$) is this
Ajtai--Koml\'os--Tusn\'ady\,/\,Ambrosio--Stra--Trevisan effect~\cite{ast}: in $d=2$ the two-sample
fluctuation scales as $\sqrt{\log M/M}$, whose finite-range log--log slope over
$M\in[6\times10^2,4\times10^3]$ is $-\tfrac12+\tfrac{1}{2\log M}\approx-0.43$ (Appendix~\ref{app:d2}),
already below the asymptotic $0.50$ and close to the observed $0.39$. \textbf{(2) The unified exponent.} Combining the measured
curse with the exact temporal Otto--Taylor bias and optimizing the pooling window reproduces the
predicted $\gamma_d(k+1)/(k+1+\gamma_d)$ (Panel~3, solid vs.\ dashed): the exponent rises with
smoothness $k$ and falls with $d$, collapsing to the location rate for $d\le2$. \textbf{(3)
Endpoint estimation ($h=0$).} An endpoint-estimation experiment --- pooling de-drifted snapshots
within an optimized bandwidth, with $h=0$ so it isolates the statistics-dominated branch
(current-distribution estimation rather than future forecasting) --- recovers the unified exponent
(Panel~2; stars in Panel~3). For $d=2$ the fitted $M$-exponents are $0.31,0.31,0.35$ for $k=0,1,2$, on the
predicted band (theory $0.33,0.40,0.43$; semi-empirical with the measured curse $0.28,0.33,0.35$)
and rising toward $k{=}2$. For $d=3$ the fits $0.19,0.18,0.22$ sit below the asymptotic prediction
($0.25,0.29,0.30$), a finite-budget effect: the curse itself is still pre-asymptotic at these $M$
($\hat\gamma_3=0.31$ vs.\ $1/3$), which lowers the whole estimation exponent. This confirms
Theorem~\ref{thm:upper0} ($k{=}0$) and is consistent with the conditional construction of
Proposition~\ref{prop:covariant} ($k\ge1$); the run is the de-drift-plus-pooling surrogate, not the
full development forecaster, so it probes the predicted exponent rather than verifying the
geometry. A genuine
positive-horizon run at $h=o(H_\#)$ would share the same exponent while remaining a forecast.

\begin{figure}[tbp]
\centering
\includegraphics[width=\textwidth]{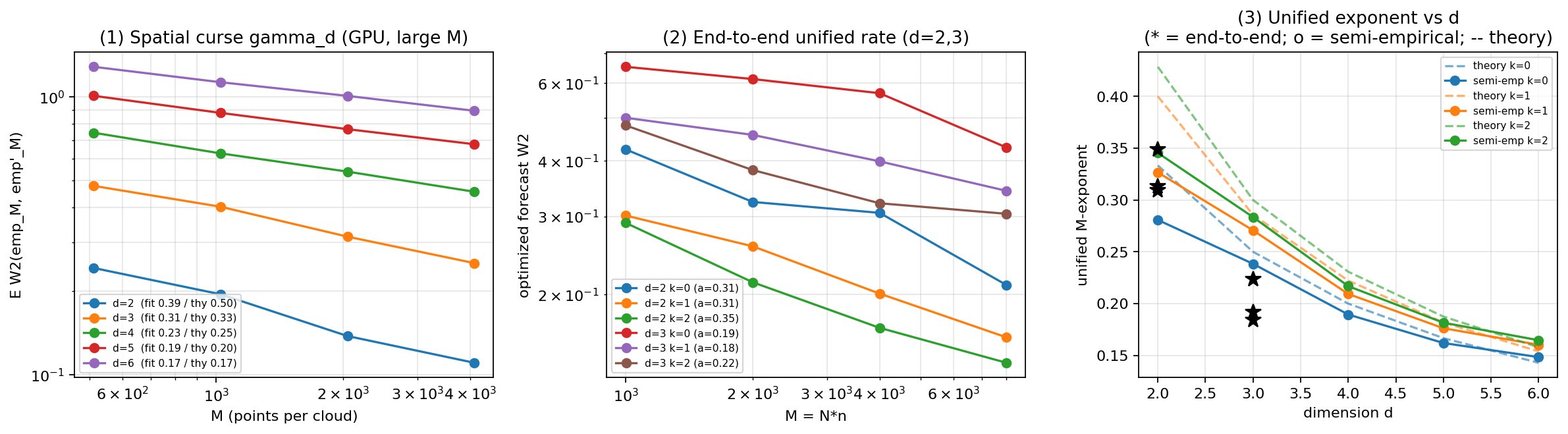}
\caption{Unified rate over $\PP$ (Theorem~\ref{thm:unified}, Conjecture~\ref{conj:upper}),
isotropic Gaussians in $\R^d$. \textbf{(1)} empirical-$W_2$ fluctuation (two-sample proxy for the estimation risk) vs.\ $M$, fitted curse
exponents against $M^{-\min(1/d,1/2)}$ for $d=2,\dots,6$ (debiased Sinkhorn divergence; an exact
EMD solver agrees to $0.01$ for $d\le4$). \textbf{(2)} endpoint estimation (de-drift $+$
pooling, optimized bandwidth, $h=0$) isolating the statistics-dominated branch, fitted $M$-exponent
rising with $k$ for $d=2$; the $d=3$ fits are pre-asymptotic at these budgets. \textbf{(3)} the unified $M$-exponent vs.\ $d$ for
$k=0,1,2$: semi-empirical (measured curse $+$ exact bias, solid) against theory
$\gamma_d(k+1)/(k+1+\gamma_d)$ (dashed); stars mark the endpoint-estimation fits
(on the band for $d=2$, pre-asymptotic for $d=3$).}
\label{fig:unified}
\end{figure}

\paragraph{Held-out predictive validation (Figure~\ref{fig:heldout}).} To rule out post-hoc
tuning of the bias--variance trade-off, we split a drifting field into a calibration half and a
held-out test half. From the calibration half \emph{alone} we fit the two constants of the model
$\mathrm{err}^2(H)=a^2(h+H/2)^2+b^2/(NH)$ (extrapolation bias $+$ pooled-estimator variance); the
fitted drift coefficient $a=0.030$ recovers the true per-step drift $0.031$, whereas the naive
increment $\|\Delta Q\|=0.49$ instead measures sampling noise --- the finite-sample pitfall of
Section~\ref{sec:stat}. The calibrated model then \emph{predicts}, on the untouched test half, the
U-shaped bandwidth curve (median relative error $18\%$) and its interior optimum
($H_\#\approx10$ vs.\ measured $H^*=8$). The optimal pooling bandwidth of
Theorem~\ref{thm:upper0} is thus a genuine out-of-sample prediction of the theory, not a fit.

\begin{figure}[tbp]
\centering
\includegraphics[width=\textwidth]{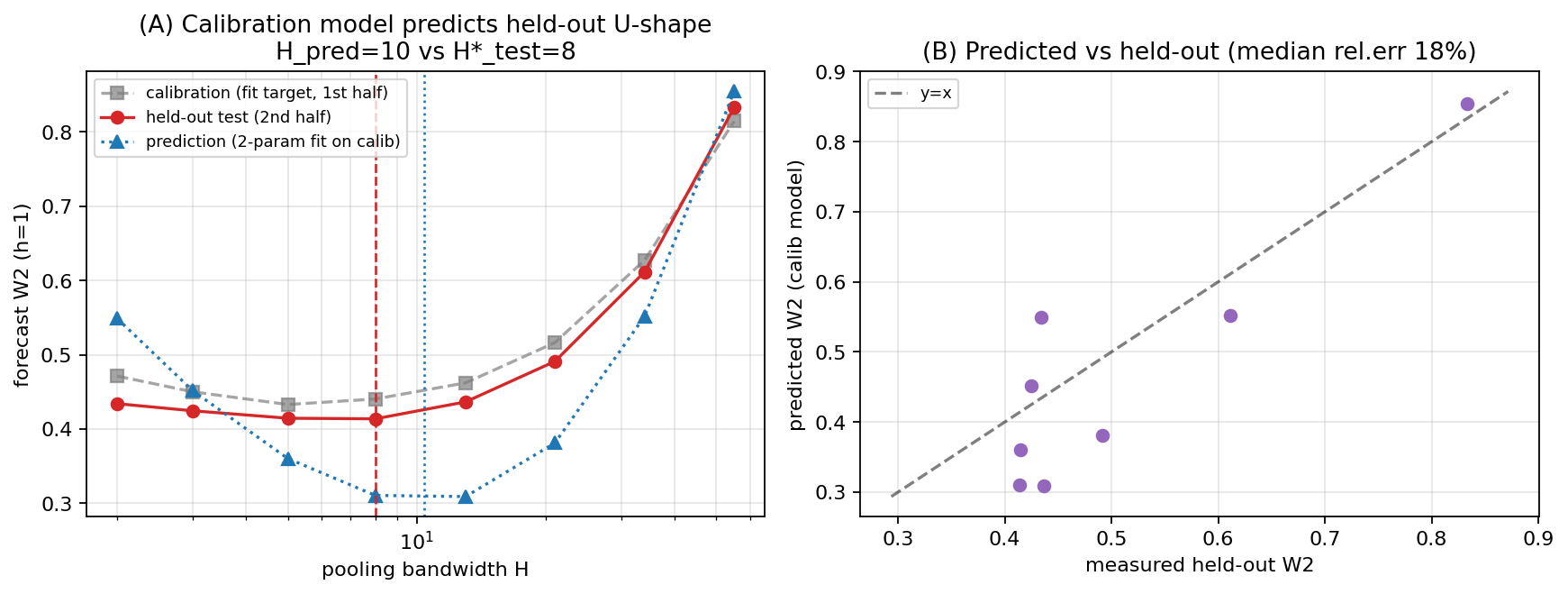}
\caption{Held-out predictive validation. \textbf{(A)} a bias--variance model with two constants
fit on the calibration half (blue) predicts the held-out test U-shape (red) and its optimal
pooling bandwidth $H_\#$; grey is the calibration fit target. \textbf{(B)} predicted vs.\ measured
held-out forecast error across the bandwidth grid (median relative error $18\%$). The optimum is
predicted out-of-sample, not fitted.}
\label{fig:heldout}
\end{figure}

\subsection{Two real series at opposite ends of the drift/noise spectrum}\label{sec:realdata}
The synthetic experiments isolate each rate under controlled conditions. We complement them with two
real distribution-valued series chosen to sit at opposite extremes of the drift-to-noise ratio,
scoring both by rolling-origin backtesting (expanding past, no look-ahead).

\paragraph{Near-stationary: S\&P~500 (Figure~\ref{fig:real}).} Daily cross-sections of log-returns of
the S\&P~500 constituents, one empirical measure $\hat\mu_t$ per trading day ($2514$ days,
$2015$--$2024$; $\approx192$ names/day), the series studied in the Wasserstein-autoregression
literature (Zhang--Kokoszka--Petersen). Two findings align with the theory. First, the
\emph{effective extrapolation order is data-dependent}: degree-$0$ persistence is the best forecaster
at every horizon, degree-$1$ is slightly worse, and degree-$2$ degrades sharply with $h$ --- on a
high-noise, near-stationary series the higher-order tangent forecaster of
Appendix~\ref{app:forecaster} extrapolates sampling noise, precisely the $k{=}0$ regime of
Theorem~\ref{thm:upper0}. Second, the \emph{moving-versus-static gap persists}: the one-step
pooled-persistence error sits at $\approx10^{-2}$ and does not fall with the sample budget, whereas
static empirical-$W_2$ estimation of a frozen law decays as $M^{-1/2}$; a finite-sample noise
reference $\tau(N)$ lies $\approx6.6\times$ below the floor, a persistent moving-versus-static gap
not explained by the finite-sample noise reference alone.

\paragraph{Strongly drifting: surface temperature (Figure~\ref{fig:temp}).} Daily mean $2$\,m
temperature over a $12\times10$ European lon--lat grid (Open-Meteo ERA5 archive, Jan--Jun~2023;
$120$ cells, $15$-day smoothed to remove synoptic weather), one cross-section $\hat\mu_t$ per day
with a $\approx+0.05^\circ$C/day seasonal drift. Here the predictions that the near-stationary S\&P
series masks become directly visible. \emph{(A)} pooled persistence has an \emph{interior} optimal
bandwidth $H^*=3$ days (Theorems~\ref{thm:stat}/\ref{thm:upper0}): too little pooling is
variance-limited, too much crosses the warming trend. \emph{(B)} the horizon exponents rise with the
forecaster order, fitted slopes $0.21,0.42,1.25$ for $k=0,1,2$ --- still below the integer $k{+}1$,
as finite-station leverage damps them (Corollary~\ref{cor:leverage}), but an order of magnitude
above the near-stationary S\&P slopes. \emph{(C)} the moving forecast floor ($\approx1^\circ$C)
again sits far above the static $M^{-1/2}$ estimation curve. We are explicit that the slopes are
damped and that the raw increment overstates the drift (here $\|\Delta Q\|$ reflects $30$-station
sampling noise, not the $\approx0.05^\circ$C/day signal); the quantities we read off are the
\emph{measured} optimal bandwidth and the slope \emph{ordering}, not a parametric rate. The slope ordering and the interior optimum persist across smoothing windows $\{1,7,15,30\}$ days,
including no smoothing, so the smoothness-order evidence is not a preprocessing artifact
(Appendix~\ref{app:smooth}). Across the two series the observable horizon slope grows monotonically with the drift-to-noise ratio (S\&P
$\approx0.01$, temperature $0.21$): a direct demonstration that the effective extrapolation order
depends on the drift-to-noise regime.

\begin{figure}[tbp]
\centering
\includegraphics[width=0.62\textwidth]{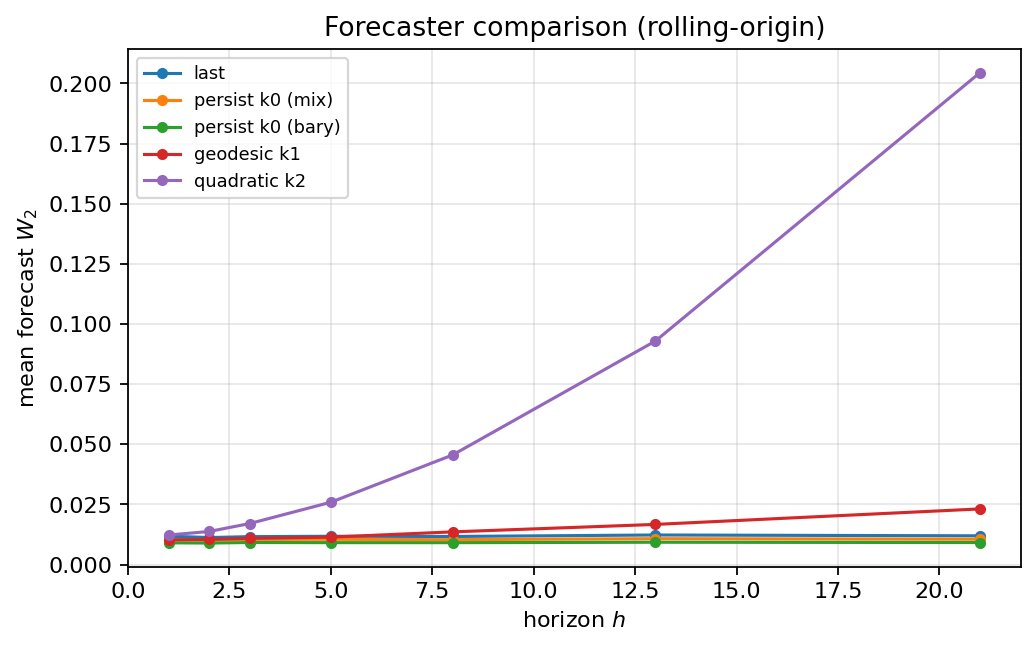}\\[4pt]
\includegraphics[width=0.62\textwidth]{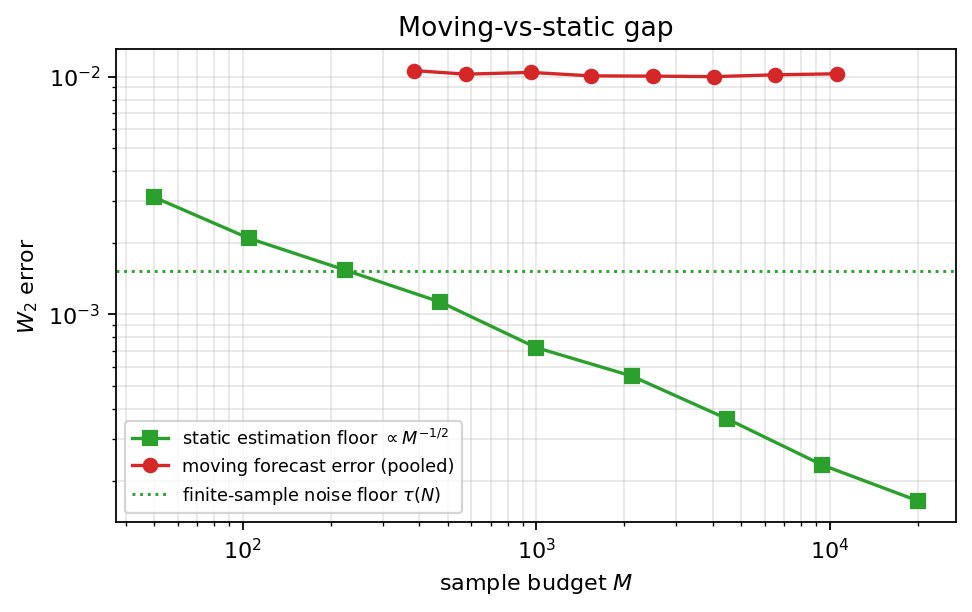}
\caption{Real-data illustration on S\&P~500 daily return cross-sections ($2514$ days, $\approx192$
names/day), rolling-origin. \textbf{(Top)} Forecast error vs.\ horizon: degree-$0$ persistence
(and the last-snapshot baseline) are best, the degree-$1$ geodesic forecaster is slightly worse,
and the degree-$2$ forecaster diverges with $h$ --- on a high-noise series higher-order
extrapolation amplifies sampling noise, consistent with a $k{=}0$ regime
(Theorem~\ref{thm:upper0}). \textbf{(Bottom)} Moving-versus-static gap: the pooled one-step
forecast error (red) stays at $\approx10^{-2}$ independently of the sample budget $M$, while static
empirical-$W_2$ estimation of a frozen law (green) decays as $M^{-1/2}$; the finite-sample noise
reference $\tau(N)$ (dotted) lies $\approx6.6\times$ below the forecast floor, a persistent
moving-versus-static gap not explained by the noise reference alone.}
\label{fig:real}
\end{figure}

\begin{figure}[tbp]
\centering
\includegraphics[width=\textwidth]{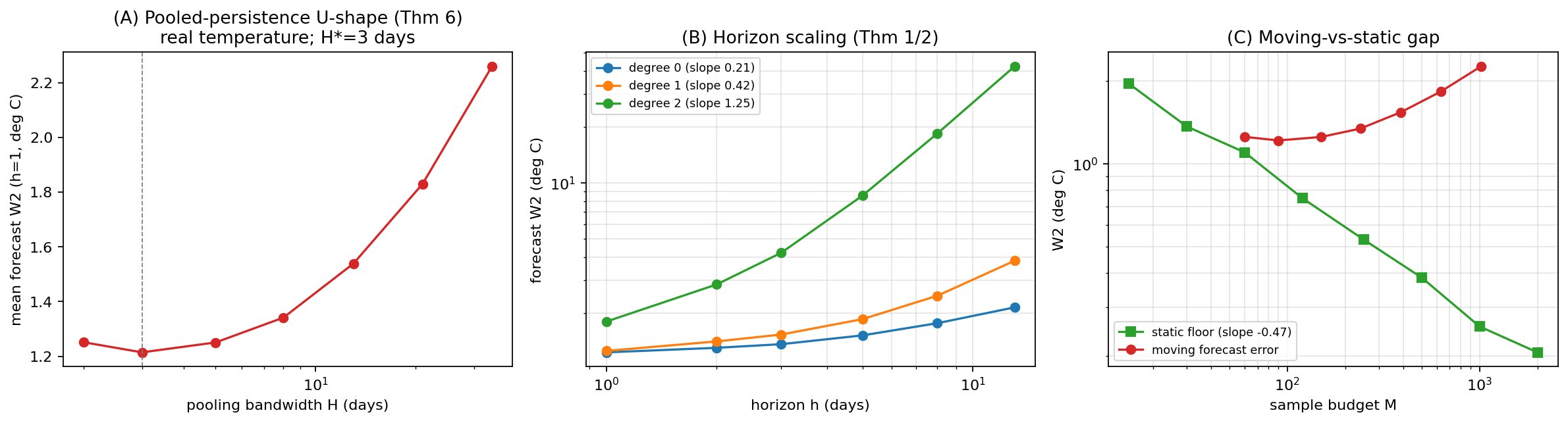}
\caption{Strongly-drifting real series: daily $2$\,m surface temperature over a European grid
(Open-Meteo ERA5 archive, Jan--Jun~2023, $15$-day smoothed). \textbf{(A)} pooled-persistence error
vs.\ bandwidth $H$, with an interior optimum $H^*=3$ days (Theorem~\ref{thm:upper0}). \textbf{(B)}
horizon scaling, fitted slopes $0.21,0.42,1.25$ for $k=0,1,2$, rising with smoothness and an order
of magnitude above the near-stationary S\&P slopes (finite-station leverage keeps them below the
integer $k{+}1$). \textbf{(C)} moving-vs-static gap: the $\approx1^\circ$C moving forecast floor
sits far above the static $M^{-1/2}$ estimation curve.}
\label{fig:temp}
\end{figure}

\section{Robustness to the temperature smoothing window}\label{app:smooth}
The real-temperature experiment of Section~\ref{sec:realdata} applies a $15$-day rolling mean to
remove synoptic weather. To verify that the reported smoothness order is not an artifact of this
preprocessing, we recompute the interior optimal bandwidth $H^*$ and the fitted horizon slopes on
the same cached field for smoothing windows of $1$ (no smoothing), $7$, $15$, and $30$ days,
holding every other setting fixed.

\begin{center}
\begin{tabular}{rccccc}
\hline
smoothing (days) & $\hat\varepsilon$ ($^\circ$C/day) & $H^*$ (days) & slope $k{=}0$ & $k{=}1$ & $k{=}2$\\
\hline
$1$ (none) & $2.50$ & $2$ & $0.18$ & $0.50$ & $1.24$\\
$7$        & $2.16$ & $3$ & $0.21$ & $0.50$ & $1.27$\\
$15$       & $2.09$ & $3$ & $0.21$ & $0.42$ & $1.25$\\
$30$       & $1.98$ & $3$ & $0.20$ & $0.38$ & $1.25$\\
\hline
\end{tabular}
\end{center}

\noindent The horizon slopes --- the actual evidence for temporal smoothness order --- are
nearly invariant ($k{=}0$: $0.18$--$0.21$; $k{=}1$: $0.38$--$0.50$; $k{=}2$: $1.24$--$1.27$)
and remain an order of magnitude above the near-stationary S\&P values ($\approx0.01$) even with
\emph{no} smoothing. Smoothing lowers the day-to-day noise $\hat\varepsilon$ and sharpens the
variance-limited regime --- moving $H^*$ off the grid boundary at $w{=}1$ to a stable interior
optimum of $3$ days for $w\ge7$ --- but does not manufacture the smoothness order, which is carried
by the large seasonal drift already present in the raw field (Figure~\ref{fig:smooth}). This
non-deseasonalized field mixes a near-deterministic seasonal trend with the weather residual; the
complementary deseasonalized-residual regime (a near-stationary, S\&P-like field) is left to future
work.

\begin{figure}[tbp]
\centering
\includegraphics[width=\textwidth]{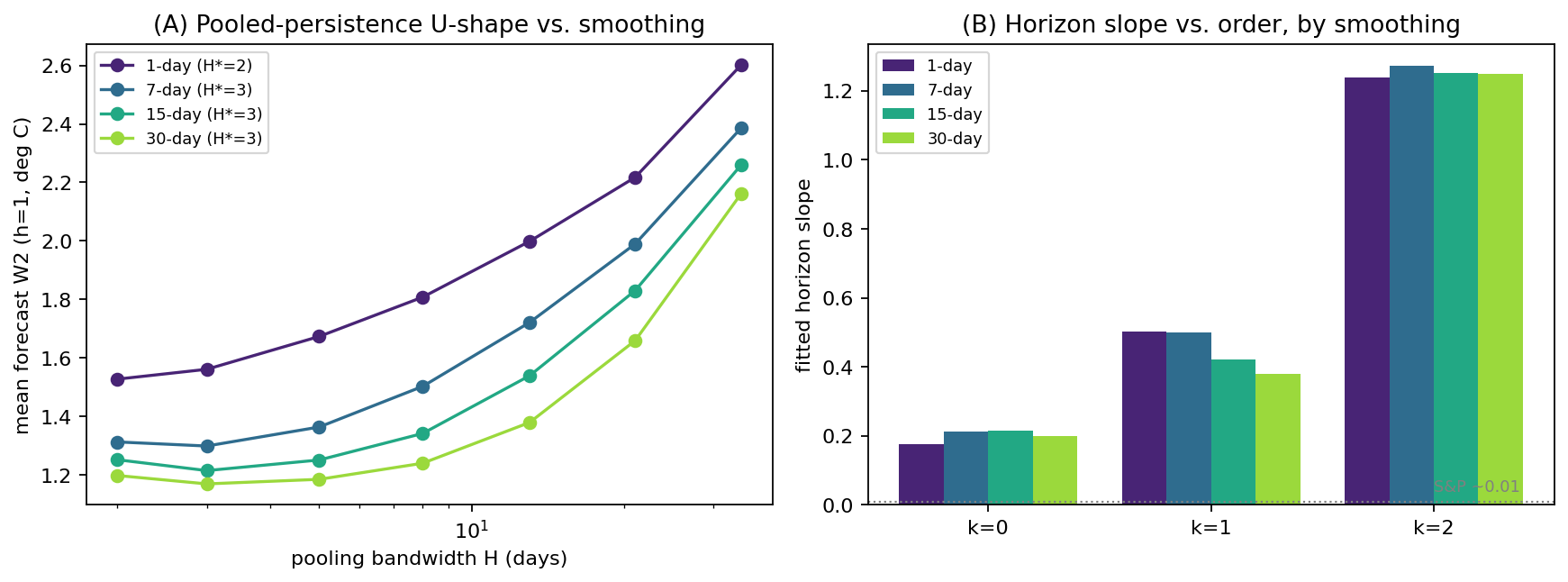}
\caption{Robustness of the real-temperature experiment to the smoothing window
$w\in\{1,7,15,30\}$ days. \textbf{(A)} pooled-persistence U-shape: the interior optimum is stable at
$H^*\approx3$ days for $w\ge7$ and only touches the grid boundary at $w{=}1$. \textbf{(B)} fitted
horizon slopes by forecaster order $k=0,1,2$: the rising-slope structure is nearly invariant
across windows and far above the near-stationary S\&P reference, so it is not a smoothing artifact.}
\label{fig:smooth}
\end{figure}

\section{The $d=2$ logarithmic correction}\label{app:d2}
In $d=2$ the empirical-$W_2$ two-sample fluctuation is not exactly $M^{-1/2}$. The
Ajtai--Koml\'os--Tusn\'ady optimal-matching result, made sharp for absolutely continuous laws by
Ambrosio--Stra--Trevisan~\cite{ast}, gives
$\E\,W_2(\hat\mu_M,\hat\mu_M')\asymp\sqrt{\log M/M}=M^{-1/2}(\log M)^{1/2}$, whose log--log slope is
\[
   \frac{\mathrm d\log\E\,W_2}{\mathrm d\log M}=-\tfrac12+\tfrac{1}{2\log M},
\]
strictly above $-\tfrac12$ and decaying only logarithmically. Over the range
$M\in[6\times10^2,4\times10^3]$ of Figure~\ref{fig:unified}(1) this local slope runs from $-0.42$ to
$-0.44$, and a single power-law fit to $\sqrt{\log M/M}$ across the range returns an effective
exponent of $0.43$ (Figure~\ref{fig:d2}) --- already well below the asymptotic $0.50$ and close to
the observed $0.39$, the residual being the Ambrosio--Stra--Trevisan constant and sub-leading terms.
The $d=2$ undershoot is therefore the expected finite-range form of the boundary log-correction, not
a breakdown of the curse exponent; for $d\ge3$ no such correction appears and the fitted exponents
track $1/d$ directly.

\begin{figure}[tbp]
\centering
\includegraphics[width=\textwidth]{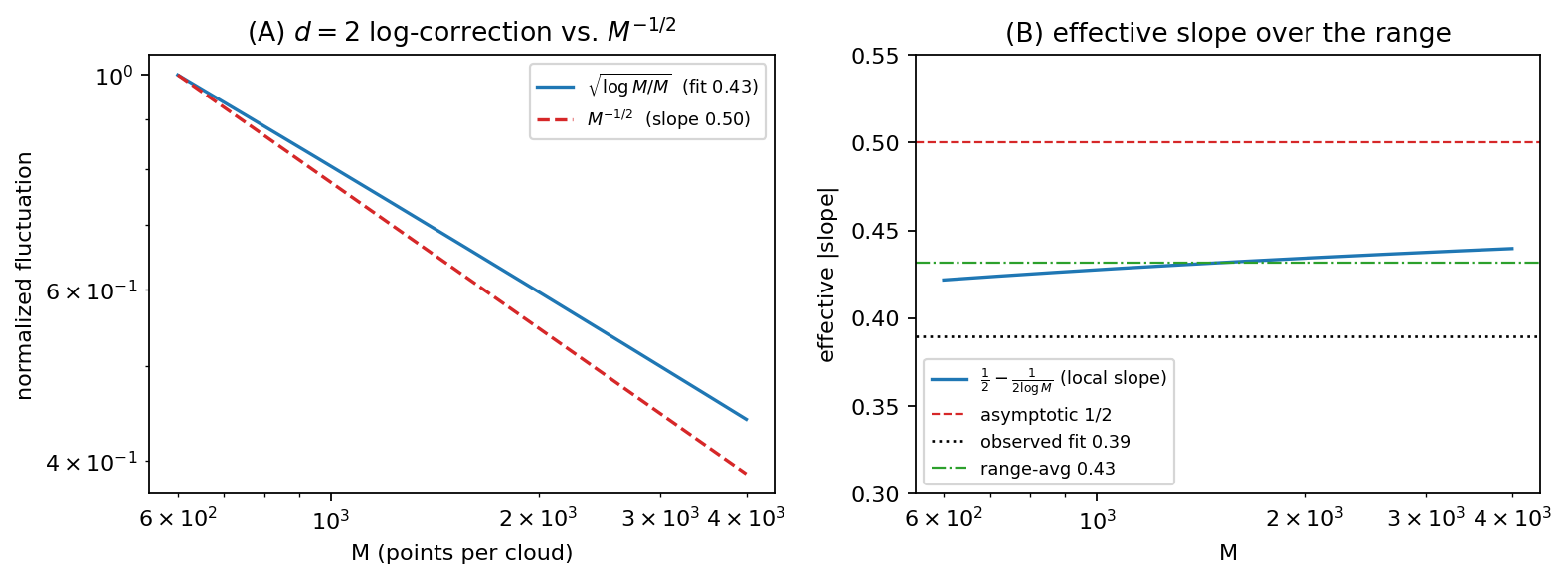}
\caption{The $d=2$ logarithmic correction. \textbf{(A)} $\sqrt{\log M/M}$ (effective single-slope
fit $0.43$) against the asymptotic $M^{-1/2}$ (slope $0.50$) over the range of
Figure~\ref{fig:unified}(1). \textbf{(B)} the local slope $\tfrac12-\tfrac{1}{2\log M}$ (running
$0.42$--$0.44$), the asymptotic $0.50$, the range-averaged $0.43$, and the observed fitted exponent
$0.39$.}
\label{fig:d2}
\end{figure}

\end{document}